\documentclass[aop,MSNbibl,dvips]{arximspdf}
\usepackage{graphicx}
\doi{10.1214/13-AOP872} %kopijuoti is PTS
\volume{43}
\issue{1}
\pubyear{2015}
\firstpage{119}
\lastpage{165}

\makeatletter
\newcommand{\rrvert}{\vert}
\newcommand{\llvert}{\vert}
\newtheorem{Theorem}{Theorem}[section]
\newproclaim{Remark}{Remark}
\newtheorem{Proposition}[Theorem]{Proposition}
\newtheorem{Corollary}[Theorem]{Corollary}
\newtheorem{Lemma}[Theorem]{Lemma}
\makeatother

\begin{document}
\begin{frontmatter}

\title{Convergence rates for loop-erased random walk and other Loewner curves}
\runtitle{Convergence rates for Loewner curves}

\begin{aug}
\author[A]{\fnms{Fredrik} \snm{Johansson Viklund}\corref{}\ead[label=e1]{viklund@math.uu.se}\thanksref{T1}}
\runauthor{F. Johansson Viklund}
\affiliation{Columbia University}
%Columbia University\\
%2990 Broadway\\
%New York, New York 10027\\
%USA\\
\address[A]{Department of Mathematics\\
Uppsala University\\
Box 480, 751 06 Uppsala\\
Sweden\\
\printead{e1}}
\end{aug}
\thankstext{T1}{Supported by the Simons Foundation, Institut Mittag--Leffler, and the
AXA Research Fund.} % Supported by the Simons Foundation.}

% HISTORY:
\received{\smonth{9} \syear{2012}}
\revised{\smonth{7} \syear{2013}}

% ABSTRACT
%
\begin{abstract}
We estimate convergence rates for curves generated by Loewner's
differential equation under the basic assumption that a convergence
rate for the driving terms is known. An important tool is what we call
the \emph{tip structure modulus}, a geometric measure of regularity for
Loewner curves parameterized by capacity. It is analogous to
Warschawski's boundary structure modulus and closely related to annuli
crossings. The main application we have in mind is that of a random
discrete-model curve approaching a Schramm--Loewner evolution (SLE)
curve in the lattice size scaling limit. We carry out the approach in
the case of loop-erased random walk (LERW) in a simply connected
domain. Under mild assumptions of boundary regularity, we obtain an
explicit power-law rate for the convergence of the LERW path toward the
radial SLE$_2$ path in the supremum norm, the curves being
parameterized by capacity.
On the deterministic side, we show that the tip structure modulus gives
a sufficient geometric condition for a Loewner curve to be H\"older
continuous in the capacity parameterization, assuming its driving term
is H\"older continuous. We also briefly discuss the case when the
curves are \emph{a priori} known to be H\"older continuous in the
capacity parameterization and we obtain a power-law convergence rate
depending only on the regularity of the curves.
\end{abstract}

% KEYWORDS

\end{frontmatter}

%s1 #&#
\section{Introduction, motivation and results}\label{sec1}
%s1.1 #&#
\subsection{Introduction}\label{sec1.1}
The Loewner equation is a partial differential equation that produces a
Loewner chain, a family of conformal mappings from a reference domain
onto a continuously decreasing sequence of simply connected domains.
The evolution is controlled by a real valued function called driving
term which acts as a parameter. Under smoothness assumptions on the
driving term, the Loewner equation can be used to generate a growing
continuous curve, by which we mean a continuous function from some
interval into the reference domain. Conversely, starting from a
suitable curve one can reverse the procedure to recover the driving
term and so there is a correspondence between \emph{Loewner curves} and
their driving terms. Following Schramm \cite{SchrammIJM}, Loewner's
equation has in recent years been successfully applied to study
conformally invariant scaling limits of certain lattice models from
statistical physics. By taking a scaled Brownian motion as the driving
term, one obtains the one-parameter family of random fractal
Schramm--Loewner evolution (SLE) curves which are essentially the only
possible conformally invariant scaling limits of cluster interfaces
with a certain Markovian property; see \cite{SchrammIJM}. Convergence
to SLE has been proved in several cases; see, for example, \cite{schrammicm}
and the references therein. The use of the Loewner equation and SLE
techniques in this context has made it possible to give precise meaning
to the (passage to the) scaling limit itself, but also to prove
conformal invariance, and to give rigorous proofs of various
predictions made by physicists. The latter is to large extent due to
the fact that the SLE processes are amenable to computation via
stochastic calculus.

In this paper, we will be interested in quantifying the relationship
between (random) rough Loewner curves with driving terms that are close
in the supremum norm. To explain our interest, let us first consider a
nonrandom setting. One can view the Loewner equation as a highly
nonlinear function from a space of driving terms to a suitable metric
space of (parameterized) curves and it is natural to ask about
continuity properties, if any. This point of view is closely related to
work by Lind, Marshall and Rohde; see \cite{MR} and \cite{LMR}. For
example, Theorem~4.1 of~\cite{LMR} proves that curves driven by H\"
older-$1/2$ driving terms with small semi-norm converge as curves if their
driving terms converge. So the ``Loewner function'' is continuous when
restricted to this collection of driving terms and our results can be
used to show that it is H\"older continuous with an explicit exponent
depending only on the semi-norm assuming it is sufficiently small. One
can also ask similar questions, restricting attention to driving terms
generating curves with some given regularity.

Our principal motivation, however, comes from the observation that
although several discrete-model curves are known to converge (as curves
up to reparameterization) to SLE curves, next to nothing appears to be
known about the rates of their convergence. (See the paper \cite{BJK}
by Bene\v{s}, Kozdron and the author for a quantitative result of
convergence of loop-erased random walk at a fixed time with respect to
Hausdorff distance when the curves are viewed as compact sets.)

Good control over convergence rates would allow SLE techniques to be
used on mesoscopic scales, that is, scales of order $\varepsilon^p$
with $p \in
(0, 1)$ where $\varepsilon$ is the lattice spacing. It is reasonable
to believe
that such results will be helpful for obtaining fine properties of
corresponding discrete models; this question was raised by Schramm in
connection with sharp estimation of critical exponents \cite
{schrammicm}. We may compare the present setting with a related model. So-called strong
approximation results such as the KMT approximation or the Skorokhod
embedding \cite{LL} yield couplings in which simple random walk and
Brownian motion paths are close with high probability, with error terms
expressed explicitly in terms of the lattice spacing. This gives a
natural way to use techniques for Brownian motion to deduce fine
properties of simple random walk that can depend on behavior on
mesoscopic scales. This approach has been used by, for example, Lawler,
Lawler and Puckette, and Bene\v{s}; see \cite{LP} and \cite{benes} and
the references therein. It thus seems that approximation results with
explicit error terms for discrete models converging to SLE could be
quite useful. Presently, all known proofs of convergence to SLE goes
via convergence of the driving terms in one way or another, so it seems
natural to take a convergence rate for the driving terms as a starting
point. We remark that the work in \cite{BJK} essentially reduces the
derivation of a convergence rate for the driving terms to the
derivation of a convergence rate for the so-called martingale
observable in rough domains. We will show that a power-law convergence
rate to an SLE curve can be derived from a power-law convergence rate
for the driving terms provided some additional quantitative geometric
information, related to crossing events, is available for the discrete
curves, along with an estimate on the growth of the derivative of the
SLE map. The approach is quite general and we believe it can be applied
to several models (even with nonsimple scaling limit curves) as soon as
the aforementioned information is available, though we carry out the
specific probabilistic estimates only in the case of loop-erased random walk.

%s1.2 #&#
\subsection{Overview, results and related work}\label{sec1.2}
Let us briefly sketch the setup and main ideas in the (chordal)
half-plane setting, though we will later work mostly in the disk. See
Section~\ref{prel} for precise definitions. Let $W, W_n\dvtx  [0,T] \to
\mathbb{R}$ be continuous functions such that
\[
\sup_{t \in[0,T]}\bigl|W(t)-W_n(t)\bigr| \le\varepsilon,
\]
where $\varepsilon>0$ is small but for the moment fixed. Let $f(t,z)\dvtx
\mathbb
{H} \to H(t)$ and $f_n(t,z)\dvtx  \mathbb{H} \to H_n(t)$ be the solutions to
the chordal Loewner equation (Loewner chains)
\[
\partial_t f(t,z) = -\partial_z f(t,z)
\frac{2}{z-U(t)},\qquad f(0,z)=z, z \in\mathbb{H}
\]
with $U(t)$ replaced by $W(t)$ and $W_n(t)$, respectively. Assume that
the Loewner chains are generated by the curves $\gamma$ and $\gamma_n$
parameterized by capacity so that for each $t$, $H(t)$ and $H_n(t)$ are
the unbounded components of $\mathbb{H} \setminus\gamma[0,t]$ and
$\mathbb{H} \setminus\gamma_n[0,t]$, respectively. (We can think of
$\gamma_n$ as the conformal image of a discrete-model curve on a
lattice approximation of a smooth domain $D$, where the mesh of the
lattice is $n^{-1}$, and the driving term of $\gamma_n$ is coupled with
a scaled Brownian motion $W$ driving the chordal SLE curve $\gamma$ so
that the driving terms are at distance at most $\varepsilon=n^{-q}$
for some
$q<1$.) Let $y > 0$; we will later choose $y=y(\varepsilon)$. Let $t
\in
[0,T]$. We can write
\begin{eqnarray*}
\bigl|\gamma(t)-\gamma_n(t)\bigr| &\le& \bigl|\gamma(t)-f\bigl(t,W(t)+iy\bigr)\bigr|
\\
&&{} + \bigl|f\bigl(t,W(t)+iy\bigr)-f\bigl(t,W_n(t)+iy\bigr)\bigr|
\\
&&{} + \bigl|f\bigl(t,W_n(t)+iy\bigr) - f_n
\bigl(t,W_n(t)+iy\bigr)\bigr|
\\
&&{} + \bigl|f_n\bigl(t,W_n(t)+iy\bigr)-
\gamma_n(t)\bigr|
\\
&=: & A_1+A_2+A_3+A_4.
\end{eqnarray*}
We wish to estimate the $A_j$ in terms of $\varepsilon$.
Suppose that there are $\beta<1$ and $c < \infty$ such that
%
%e1 #&#
\begin{equation}
\label{intro1} \bigl|f'\bigl(t, W(t)+id\bigr)\bigr| \le c d^{-\beta}\qquad\mbox{for all } d \le y.
\end{equation}
If this estimate holds, then by integrating, $A_1 \le c   y^{1-\beta
}$. (Constants may change from line to line, and are assumed to depend
only on the parameters and not on $\varepsilon,y$, etc.) By the distortion
theorem, the same bound holds for $A_2$ if $y \ge\varepsilon$. The
third term,
$A_3$, represents the distance between two solutions to the Loewner
equation having driving terms at supremum distance at most $\varepsilon
$, and
evaluated at the same point. In Section~\ref{gronwall-sect}, we will
use the reverse-time Loewner flow to estimate quantities like this. In
particular, we will see that if $\operatorname{Im}z=y$, then
\[
\bigl\llvert f(t, z) - f_n(t,z) \bigr\rrvert \le c \varepsilon
y^{-1}
\]
with $c$ depending only on $T$. Hence, $A_3 \le c   \varepsilon
y^{-1}$ and
Cauchy's integral formula implies that
\[
\bigl\llvert y\bigl|f'(t, z)\bigr| - y\bigl|f_n'(t,
z)\bigr|\bigr\rrvert \le c \varepsilon y^{-1}.
\]
From this it follows, using Koebe's estimate and (\ref{intro1}), that if
\[
\Delta_n(t, y):=\operatorname{dist} \bigl[f_n
\bigl(t,W_n(t)+iy\bigr), \partial H_n(t) \bigr],
\]
then
%
%e2 #&#
\begin{equation}
\label{intro2} \Delta_n(t, y) \le c y\bigl|f_n'
\bigl(t,W_n(t)+iy\bigr)\bigr| \le c y^{1-\beta} +c \varepsilon
y^{-1};
\end{equation}
see Proposition~\ref{rough}. (Note that we have made no explicit
assumption on the behavior of $|f_n'|$.) Now choose $y(\varepsilon
)=\varepsilon^p$, for
some $p \in(0,1)$. Then
\[
A_1+A_2+A_3 \le c \varepsilon^{p(1-\beta)}
+ c \varepsilon^{1-p}
\]
and it remains to bound $A_4$. Clearly, $A_4 \ge\Delta_n(t,
\varepsilon^p)$
but we would like an upper bound in terms of $\Delta_n(t, \varepsilon
^p)$. To
proceed, some additional information about the boundary behavior of
$f_n$ is necessary.

For this, we will use what we call the tip structure modulus, a
geometric gauge of the regularity of a Loewner curve in the capacity
parameterization, that is, for our problem, the analog of Warschawski's
\cite{W} measure with a similar name. Let $\delta>0$ and consider
$\mathcal{S}_{t,\delta}$, the set of all crosscuts of $H_n(t)$ of
diameter at most $\delta$ that separate the tip, $\gamma_n(t)$, from
$\infty$ in $H_n(t)$. Each crosscut $\mathcal{C}\in\mathcal
{S}_{t,\delta}$
separates from $\infty$ in $H_n(t)$ a piece $\gamma_{\mathcal{C}}$
of $\gamma
_n[0,t]$ obtained by tracing $\gamma_n$ backward from $\gamma_n(t)$
until~$\overline{\mathcal{C}}$ is first hit. (If $\gamma_n$ and
$\overline{\mathcal{C}
}$ do not intersect, we set $\gamma_\mathcal{C}=\gamma$.) We then
define the
tip structure modulus, $\eta_{\mathrm{tip}}(\delta)$, of
$\gamma_n(t), t \in
[0,T]$, to be the maximum of $\delta$ and
\[
\sup_{t \in[0,T]} \sup_{\mathcal{C}\in\mathcal{S}_{t,\delta}} \operatorname{diam}
\gamma _\mathcal{C}.
\]
(See Section~\ref{geom} for a precise definition.)
Roughly speaking, $\eta_{\mathrm{tip}}(\delta)$ is the maximal
distance the
curve travels into a ``bottle'' with ``bottleneck'' opening smaller
than $\delta$ viewed from the point toward which the curve is growing.
(Similar conditions have been used before; see below.) In
Proposition~\ref{nov4.1}, we show that
%
%e3 #&#
\begin{equation}
\label{intro3} \bigl|f_n\bigl(t,W_n(t)+iy\bigr)-
\gamma_n(t)\bigr| \le c_1\eta_{\mathrm{tip}}
\bigl(c_2 \Delta _n(t, y) \bigr),
\end{equation}
where $\eta_{\mathrm{tip}}$ is the tip structure modulus for
$\gamma_n$.
Consequently, if we have a power-law bound on the tip structure modulus
evaluated at $c   \Delta_n(t, \varepsilon^p)$, that is, if
\[
\eta_{\mathrm{tip}}\bigl(c \Delta_n\bigl(t, \varepsilon^p
\bigr)\bigr) \le c' \bigl(\Delta_n\bigl(t,
\varepsilon^p\bigr)\bigr)^r
\]
for some $r \in(0,1)$, then by (\ref{intro2})
\[
A_4 \le c \varepsilon^{p(1-\beta)r} + c \varepsilon^{(1-p)r}.
\]
We stress that the estimate on $\eta_{\mathrm{tip}}$ is only
required to hold
on the scale of $\Delta_n(t, \varepsilon^p)$ and note that the
failure of the
existence of such a bound on $\eta_{\mathrm{tip}}$ implies
certain crossing
events for the curve. If the estimates hold uniformly in $t \in[0,T]$,
then we have obtained a power-law bound in terms of $\varepsilon$ on
$\sup_{t
\in[0,T]} |\gamma(t) - \gamma_n(t)|$ and we can then conclude by
optimizing over exponents.

To implement these ideas in a particular setting, we need to show that
the assumptions we used are satisfied uniformly in $t \in[0,T]$, with
high probability in terms of $\varepsilon$. If a convergence rate for the
driving terms (or martingale observable in rough domains) is known,
then we believe it is possible to derive the remaining required
information from existing results in the literature without too much
effort, and we derive the needed SLE derivative estimates, from
estimates in \cite{JVL}, in this paper. Indeed, as already mentioned,
the event that the geometric condition fails implies annuli crossing
events that are fairly well understood for the models known to converge
to SLE.

The organization of the paper is as follows. In Section~\ref
{gronwall-sect}, we discuss some preliminaries and prove the
quantitative comparison estimates for solutions to the Loewner
equation. These estimates might be of some independent interest; see,
for example, \cite{JVRW}. We also consider a natural case when the
curves are \emph{a priori} known to be H\"older continuous in the
capacity parameterization and derive a power-law convergence rate
depending only on the regularity of the curves. See Corollaries~\ref
{lip}~and~\ref{holder2}.

In Section~\ref{geom}, we define the tip structure modulus and prove
the estimates implying (\ref{intro3}). Then in Theorem~\ref
{geom-holder}, we show that if a Loewner curve $\gamma$ has the
property that there is $M< \infty$ such that $\eta_{\mathrm{tip}}(\delta) \le M
\delta,   \delta< \delta_0$, and the driving term is H\"older
continuous, then $\gamma$ is also H\"older continuous in the capacity
parameterization with exponent depending only on $M$ and the exponent
for the driving term. A linear bound on the structure modulus is a
natural analog of the John condition for simply connected domains; see,
for example, Chapter~5 of \cite{Pom92}. Theorem~\ref{geom-holder} can
thus be viewed as the analog for Loewner curves of the well-known fact
that a John domain is also a H\"older domain \cite{Pom92}.

In Section~\ref{lerw-sect}, we apply the above ideas to obtain a
power-law estimate on the convergence rate to radial SLE$_2$ for the
loop-erased random walk (LERW) path. Here is an informal version of
the result; see Theorem~\ref{lerw-thm} for a precise statement. Let
$D_n$ be a $n^{-1}\mathbb{Z}^2$ grid-domain approximation of a fixed
simply connected Jordan domain $D$ containing $0$ and with $\mathcal
{C}^{1+\alpha}$ boundary and inner radius from $0$ equal to $1$. (The
proof works for the larger class of quasidisks \cite{Pom92}, but we
then get a slower convergence rate which depends on the constant in
the Ahlfors three-point condition for $D$.) Let $\gamma_n$ be the
time-reversal of LERW in $D_n$ from $0$ to $\partial D_n$ and let
$\tilde{\gamma}_n$ be its image in $\mathbb{D}$ under the conformal
map $\psi
_n\dvtx  D_n \to\mathbb{D}$ with the usual normalization. Let $\tilde{\gamma}$ be
the radial SLE$_2$ path in $\mathbb{D}$ started uniformly on $\partial
\mathbb{D}$.
Our main result can now be given as follows.

\begin{Theorem*}
For each $n$ sufficiently large, there is a coupling of $\tilde{\gamma}_n$ with $\tilde{\gamma}$ such that
\[
\mathbb{P} \Bigl\{ \sup_{t \in[0, \sigma]}\bigl|\tilde{\gamma
}_n(t)-\tilde{\gamma}(t)\bigr| > \varepsilon_n^{1/41}
\Bigr\} < \varepsilon_n^{1/41},
\]
where both curves are parameterized by capacity, $\varepsilon
_n=n^{-1/24}$ is
the convergence rate of the driving terms from \cite{BJK}, and $\sigma$
is a stopping time. The same estimate holds for the preimages of the
curves in $D_n$.
\end{Theorem*}

[The stopping time $\sigma=\sigma(\varepsilon,T)$, which is needed for
technical reasons, can be taken as the minimum of some fixed $T<\infty$
and the first time such that the forward SLE$_2$ flow of $-\tilde
{\gamma
}(0)$ is smaller than some given $\varepsilon>0$. We have $\lim_{\varepsilon\to0}
\sigma(\varepsilon,T) = T$ almost surely, see Appendix~\ref
{sle-sect}.] This
quantifies the convergence result \cite{LSW}, Theorem~3.9, of Lawler,
Schramm and Werner. As indicated, the proof considers the couplings of
\cite{BJK} in which if $s < 1/24$, then with probability at least
$1-n^{-s}$ the estimate $\sup_{t \in[0,T]}|W_n(t)-W(t)| < n^{-s}$
holds. Here, $W_n$ is of the LERW in $D_n$ and $W$ is a Brownian motion
with speed $2$ on $\partial\mathbb{D}$. Using the Brownian motion as driving
term in the Loewner equation, we have a coupling of the LERW image and
SLE$_2$ for each $n$, with their driving terms close. To prove
Theorem~\ref{lerw-thm}, we then show that the above reasoning can be
carried out on an event with large probability in terms of $n$. Some
work is required to establish the needed geometric condition for the
LERW path; see Proposition~\ref{lerw-sm}.

In Appendix~\ref{sle-sect}, we derive an estimate on the probability
(in terms of $y$) that a bound of the type (\ref{intro1}) holds for
\emph{radial} SLE from a corresponding estimate for chordal SLE from
\cite{JVL}.

Finally, in Appendix~\ref{grid-sect} we discuss a convergence rate
result for a sequence of grid-domain approximations of a quasidisk
which allows us to directly ``transfer'' the required geometric
condition to $\mathbb{D}$.

Besides classical articles by Ahlfors,
Warschawski, Becker, Pommerenke and others, which develop (Euclidean)
geometric conditions for regularity estimates on Riemann maps (see, e.g.,
\cite{W,BP,NP1,NP2,SS} and the references therein), there
are close connections between the results and methods of this paper and
more recent work. Let us highlight some. We mentioned the work by Lind,
Marshall and Rohde \cite{LMR} and by Marshall and Rohde \cite{MR}; see
also Wong's paper \cite{wong}. The paper by Aizenman and Burchard
\cite
{AB} characterizes tightness for probability measures on a space of
(discrete model) curves modulo reparameterization in terms of estimates
on probabilities of annuli crossing events. The event that the
geometric condition fails is contained in a union of crossing events of
this type and this is what allows for estimation of probabilities.
Kemppainen and Smirnov consider related questions and use similar
conditions in \cite{KS} and a quantity somewhat similar to the tip
structure modulus has been used by Lind and Rohde in \cite{LR}.

%s2 #&#
\section{Preliminaries and the deterministic Loewner equation}\label{sec2}\label{prel}
%s2.1 #&#
\subsection{Preliminaries}\label{sec2.1}
We start by setting some notation. We will write $\mathbb{D}= \{z \in
\mathbb
{C} \dvtx  |z| < 1\}$ for the unit disk in the complex plane. This is the
basic \textit{reference domain}, although we will occasionally also
consider the upper half-plane $\mathbb{H}= \{z \in\mathbb{C}\dvtx
\operatorname{Im}z > 0\}$.
Let $D \ni0$ be a simply connected domain. By the Riemann mapping
theorem, there exists a unique conformal map $\psi\dvtx  D \to\mathbb{D}$ with
$\psi(0)=0$ and $\psi'(0)>0$. If we do not state otherwise, we will
always assume that uniformizing conformal maps like $\psi$ are
normalized in this way.

A \textit{crosscut} $\mathcal{C}$ of a simply connected domain $D$ is an open
Jordan arc in $D$ such that $\overline{\mathcal{C}} = \mathcal
{C}\cup\{\zeta,\eta\}$
with $\zeta,\eta\in\partial{D}$. A crosscut partitions $D$ into
exactly two disjoint components; see Chapter~2 of \cite{Pom92}.

A (parameterized) \textit{curve} $\gamma$ is a continuous function
$\gamma
(t)\dvtx I \to\mathbb{C}$ defined on some interval $I$ which we will
usually assume to be $[0,T]$ for some fixed $T > 0$. Given two curves
$\gamma_1, \gamma_2$ defined on the same interval, we measure their
distance by the supremum norm
\[
\sup_{t \in[0,T]}\bigl|\gamma_1(t)-\gamma_2(t)\bigr|.
\]
Let $\gamma\dvtx [0,T] \to\overline{\mathbb{D}}$ be a curve with $\gamma
(0) \in
\partial\mathbb{D}, 0 \notin\gamma[0,T]$, and for $t \in[0,T]$,
let $D_t$
be the connected component of $0$ of $\mathbb{D}\setminus\gamma
[0,t]$. We
say that $\gamma$ is \textit{parameterized by capacity} if the normalized
conformal maps $g_t \dvtx  D_t \to\mathbb{D}$ satisfy $g'_t(0)=e^t$ for $t
\in
[0,T]$. (Clearly, not all curves in $\overline{\mathbb{D}}$ can be
parameterized in this way.) A~reparameterization of a curve $\gamma$ is
a new curve $\tilde{\gamma}$ obtained by $\tilde{\gamma}(t)=\gamma
\circ\alpha(t)$, where $\alpha(t) \dvtx  [0,\widetilde{T}] \to[0,T]$ is a
strictly increasing and continuous function. We will often, when no
confusion is possible, treat a curve and its reparameterizations as the
same. A ($\mathbb{D}$-) \textit{Loewner curve} is a curve $\gamma$ in
$\overline{\mathbb{D}}$ as above, parameterized by capacity, for which the following
continuity condition holds: for every $\varepsilon>0$ there exists
$\delta> 0$
such that for all $s,t \in[0, T]$ with $0 < t-s < \delta$ there is a
crosscut $\mathcal{C}$ with $\operatorname{diam}\mathcal{C}<
\varepsilon$ that separates $K_{t} \setminus
K_s$ from $0$ in $D_t$, where $K_t=\overline{\mathbb{D}\setminus D_t}$.
Intuitively, a $\mathbb{D}$-Loewner curve $\gamma$ is a continuous
curve such
that: the conformal radius from $0$ of the complement of the curve is
strictly and continuously decreasing, it has no transversal
self-crossings, and the tip $\gamma(t)$ is always ``visible'' from $0$.
For example, if $\gamma$ is piecewise smooth with no double points and
is contained in $\mathbb{D}$ for $t \in(0,T]$, then it is a Loewner
curve. By
Theorem~1 of \cite{Pom66}, the $\mathbb{D}$-Loewner curves are
exactly the
curves that can be described using the radial Loewner equation driven
by a continuous driving term, as discussed in the next section. We will
also consider (chordal) Loewner curves in $\mathbb{H}$ which are
defined in a
similar manner; we refer to Chapter~4 of \cite{lawler-book} for more
information. We just note that in this case it is convenient to
parameterize $\gamma$ by the so-called half-plane capacity, that is, so
that the conformal maps $g_t \dvtx  H_t \to\mathbb{H}$, where $H_t$ is the
unbounded connected component of $\mathbb{H}\setminus\gamma[0,t]$, satisfy
$g_t(z)=z+2t/z + o(1/|z|)$ at $\infty$. (In this case, the
normalization is at a boundary point, and the tip of the curve is to be
``visible'' from this point at all times.)

We will often write ``constants'' depending on parameters as
$c=c(a,b)$, etc. It is then to be understood that $c$ depends only on
these parameters.

%s2.2 #&#
\subsection{Loewner equations}\label{sec2.2}
We will be interested in two versions of Loew\-ner's differential
equation. We define radial and chordal Loewner vector fields by
\[
\Phi_\mathbb{D}(z,\zeta)=-z\frac{\zeta+z}{\zeta-z},\qquad\Phi _\mathbb
{H}(z,\xi)=-\frac{2}{z-\xi}.
\]
The radial and chordal Loewner equations are then given by
%
%e4 #&#
\begin{equation}
\label{crpde} \partial_t f(t,z) = \partial_zf(t,z)
\Phi_X\bigl(z, W(t)\bigr), \qquad f_0(z)=z, z \in
X,
\end{equation}
$X=\mathbb{D}$ and $X=\mathbb{H}$, respectively. (We will sometimes
refer to these equations the $\mathbb{D}$- and $\mathbb{H}$-Loewner
PDEs and their
solutions as $\mathbb{D}$- and $\mathbb{H}$-Loewner chains, etc.)
Here, $W\dvtx  [0, \infty
) \to\partial X$ is a (continuous) function called the driving term.
In the radial case, we will sometimes write the driving term as
$W(t)=e^{i \xi(t)}$ for a real valued function $\xi$ which, when no
confusion is possible, for brevity is also referred to as the driving term.

Let us discuss a few properties in the radial setting. (Similar results
hold for the chordal version.) For each $t_0 \ge0$, the solution
$f(t_0, \cdot)\dvtx  \mathbb{D} \to D_{t_0} $ is a conformal map onto a
simply connected domain $D_{t_0} \subset\mathbb{D}$. The family
$(f(t,z))_{t \ge0}$ of conformal mappings is called a \textit{Loewner
chain}. A \textit{Loewner pair} $(f, W)$ consists of a function $f(t,z)$
and a (continuous) function $W(t),   t \ge0$, such that $f$ is the
solution to the Loewner equation with $W$ as driving term. Under some
rather mild regularity assumptions on $W$ [e.g., that $W$ is H\"
older-($1/2+\varepsilon$) for some $\varepsilon>0$], there exists a
curve $\gamma(t)$
such that $D_t$ is the component of the origin of $\mathbb{D}
\setminus
\gamma[0,t]$, and in this case we say that the Loewner chain is
generated by the Loewner curve~$\gamma$. Conversely, given a Loewner
curve, one can associate via the Loewner equation a unique driving term
such that the Loewner chain $(f_t)$ in the Loewner pair $(f, W)$ is
generated by $\gamma$.
In fact, the driving term is the preimage in $\partial\mathbb{D}$ of
the tip of the growing curve. In terms of the inverse relationship, we have
%
%e5 #&#
\begin{equation}
\label{dec9.1} \gamma(t)=\lim_{d \to0+}f\bigl(t, (1-d)W(t)\bigr).
\end{equation}
A sufficient condition for $(f,W)$ to be generated by a curve $\gamma$
is that the limit (\ref{dec9.1}) exists for all $t \ge0$ and that $t
\mapsto\gamma(t)$ is continuous; see Theorem~4.1 of \cite{RS}. The
parameterization of $\gamma$ given by (\ref{dec9.1}) is the capacity
parameterization.

We will use the notation
$f_t(z) = f(t,z)$, $f'=\partial_z f$ and $\dot{f}=\partial_t f$.

%le2.1 #&#
\begin{Lemma} \label{lemma34}
There exists a constant $c_0 <\infty$ such that the following holds.
Let $X \in\{\mathbb{D}, \mathbb{H}\}$. Suppose that $f_t$ satisfies
the $X$-Loewner PDE and that $\operatorname{dist}(z, \partial X) = d$. Then
for $s \geq0$
%
%e6 #&#
\begin{equation}
\label{nov19.9} e^{-c_0s/d^2} \bigl|f_t'(z)\bigr|
\leq\bigl|f_{t+s}'(z)\bigr|\leq e^{c_0s/d^2}
\bigl|f_t'(z)\bigr|
\end{equation}
and
%
%e7 #&#
\begin{equation}
\label{mar14.7} \bigl|f_{t+s}(z) - f_t(z)\bigr| \leq c_0
 \,d \bigl|f_t'(z)\bigr| \bigl(e^{c_0s/d^2}-1\bigr).
\end{equation}
\end{Lemma}

\begin{pf}
See Lemma~3.5 of \cite{JVL} for the proof in the chordal case. The
radial case is proved in the same way.
\end{pf}

For H\"older continuous driving terms, the existence of the curve and
its regularity in the capacity parameterization is completely
determined by the local behavior at the tip, that is, the growth of the
derivative of the conformal map close to the preimage of the tip. The
following result is a version of Proposition~3.9 of \cite{JVL}, but
allows for a less regular driving term.

%
%pr2.2 #&#
\begin{Proposition}\label{der}
Let $(f,W)$ be a $\mathbb{D}$-Loewner pair and assume that $W(t)=e^{i
\xi(t)}$
where $\xi(t)$ is H\"older-$\alpha$ on $[0,T]$ for some $\alpha\le
1/2$. Then the following holds. Suppose there are $c < \infty$,
$d_0>0$, and $0 \le\beta<1$ such that
%
%e8 #&#
\begin{equation}
\label{der-eq} \sup_{t \in[0,T]} \,d\bigl|f'_t
\bigl((1-d)W(t)\bigr)\bigr| \le c d^{1-\beta}\qquad \forall d \le
d_0.
\end{equation}
Then $(f,W)$ is generated by a curve that is H\"older-$\alpha(1-\beta)$
continuous on $[0,T]$.
The analogous statement holds for $\mathbb{H}$-Loewner pairs.
\end{Proposition}

%re1 #&#
\begin{Remark*}
At $t=0$, we have $f_0'(z)=1$ so we can never do better than $\beta=
0$ in (\ref{der-eq}). However, for $t \ge\varepsilon$, we can have
$-1 \le
\beta< 0$ and in this case the curve will be H\"older-$\alpha(1-\beta)$
(which is then larger than $\alpha$) for $t \in[\varepsilon,T]$ but
only H\"
older-$\alpha$ on $[0,T]$.
\end{Remark*}

\begin{pf*}{Proof of Proposition~\ref{der}}
The bound on the derivative implies that the limit
\[
\gamma(t)=\lim_{d \to0+} f_t\bigl((1-d)W(t)\bigr)
\]
exists for every $t \in[0,T]$ and since the convergence is uniform
$\gamma(t)$ is a continuous function. Let $s>0$ and set $d=s^\alpha$.
If $t, t+s \in[0,T]$, we have
\begin{eqnarray*}
\bigl|\gamma(t+s)-\gamma(t)\bigr| &\le& \bigl|\gamma(t+s) - f_{t+s}\bigl((1-d)W(t+s)
\bigr)\bigr|
\\
&&{} + \bigl|f_{t+s}\bigl((1-d)W(t+s)\bigr)- f_{t+s}\bigl((1-d)W(t)
\bigr)\bigr|
\\
&&{} + \bigl|f_{t+s}\bigl((1-d)W(t)\bigr)- f_{t}\bigl((1-d)W(t)
\bigr)\bigr|
\\
&&{} + \bigl|\gamma(t) - f_{t}\bigl((1-d)W(t)\bigr)\bigr|.
\end{eqnarray*}
If $t >0$, then the estimate (\ref{der-eq}) implies that the first and
last terms are bounded by a constant times $d^{1-\beta}=s^{\alpha
(1-\beta)}$. By assumption $|\xi(t+s)-\xi(t)| \le c s^{\alpha} = c d$,
so the distortion theorem implies that
\[
\bigl|f_{t+s}\bigl((1-d)W(t+s)\bigr)- f_{t+s}\bigl((1-d)W(t)
\bigr)\bigr| \le c d^{1-\beta}.
\]
Finally, since $s=d^{1/\alpha}$ and $\alpha\le1/2$, (\ref{mar14.7}) implies
\[
\bigl|f_{t+s}\bigl((1-d)W(t)\bigr)- f_{t}\bigl((1-d)W(t)\bigr)\bigr|
\le c d^{1-\beta}.
\]
Since $d|f'_0((1-d)W(0))|=d$ and so cannot decay faster than linearly,
we get the stated exponent on $[0,T]$.
\end{pf*}

%
%s2.3 #&#
\subsection{An estimate for the reverse-time Loewner equation}\label{gronwall-sect}
We want to compare solutions to the Loewner equation corresponding to
driving terms which are close in the supremum norm.
We will use the \textit{reverse-time Loewner equation}: let $T <
\infty
$ and let $(f_j, W_j), j=1,2$, be Loewner pairs. Let $t_0 \in(0, T]$
be fixed.
Consider solutions $h_j(t,z; t_0)=h_j(t,z)$ to the reverse-time Loewner equation
%
%e9 #&#
\begin{equation}
\label{reverselocal} \partial_t h_j(t,z) = \Phi_X
\bigl(h_j, U_j(t)\bigr), \qquad h_j(0,z)=z,
\end{equation}
where $X$ equals $\mathbb{D}$ and $\mathbb{H}$ in the radial and
chordal case, respectively. We say that $U_j$ is the driving term for
(\ref{reverselocal}). If we take $U_j(t)=W_j(t_0-t)$ we have the
well-known identity
\[
h_j(t_0,z; t_0) = f_j(t_0,z),
\qquad z \in X, j=1,2,
\]
where $f_j(t,z)$ solves the Loewner PDE (\ref{crpde}) with $W_j(t)$ as
driving term. These equalities only hold at the special time $t=t_0$;
the families of conformal mappings $(h_j(\cdot, z))$ and $(f_j(\cdot,
z))$ are in general different. Solutions $t \mapsto h(t,z)$ to (\ref
{reverselocal}) flow away from $\partial X$ as $t$ increases when $z
\in
X$ and this implies that if $z \in X$ is fixed then the solution $t
\mapsto h(t,z)$ exists for all $t \ge0$.

Let $\varepsilon$ and $\nu$ be given nonnegative numbers. Let
$z_1,z_2 \in X$
be given and suppose that
\[
\sup_{t \in[0, T]}\bigl|W_1(t)-W_2(t)\bigr| \le
\varepsilon, \qquad|z_1-z_2| \le\nu\varepsilon.
\]
Set
\[
H(t)=h_1(t,z_1)-h_2(t,z_2),
\]
where the $h_j$ are assumed to solve the reverse-time Loewner equations
(\ref{reverselocal}) driven by
\[
\widetilde{W}_j(t):=W_j(t_0-t), \qquad j=1,2.
\]
Then $H(t_0) = f_1(t_0,z_1) - f_2(t_0,z_2)$.
We differentiate with respect to $t$ and use (\ref{reverselocal}) to
obtain the linear differential equation
\[
\dot{H}(t)-H(t)\psi_X(t)=\bigl(\widetilde W_2(t)-\widetilde
W_1(t)\bigr)\xi_X(t),
\]
where
\begin{eqnarray*}
\psi_{\mathbb{D}}(t) & =&\frac{h_1h_2 - \widetilde W_1 \widetilde W_2-
(1/2)(h_1+h_2)(\widetilde W_1+\widetilde W_2)}{(h_1-\widetilde W_1)(h_2-\widetilde W_2)},
\\
\xi_{\mathbb{D}}(t) & =&\frac{h_1^2+h_2^2}{2(h_1-\widetilde W_1)(h_2-\widetilde W_2)}
\end{eqnarray*}
and
\begin{eqnarray*}
\psi_\mathbb{H}(t) &=&\frac{2}{(h_1-\widetilde W_1)(h_2-\widetilde W_2)},
\\
\xi_\mathbb{H}(t) & =&\psi_\mathbb{H}(t).
\end{eqnarray*}
Here, we have suppressed the dependence on $t$ in the right-hand sides.
We can integrate the differential equation and with $u(t) = \exp\{
-\int_0^t \psi_X(s)    \,ds\}$ we find
\[
H(t)=u(t)^{-1} \biggl(H(0)+\int_0^t
(\widetilde W_2-\widetilde W_1)u \xi_X  \,ds
\biggr).
\]
Hence, for $0 \le t\le t_0$,
%
%e10 #&#
\begin{equation}
\label{gron} \bigl|H(t)\bigr| \le\bigl|H(0)\bigr| e^{\int_0^{t} \operatorname{Re}\psi_X(s)    \,ds} + \int_0^{t}|
\widetilde W_2-\widetilde W_1| e^{\int_s^{t} \operatorname{Re}\psi
_X(r)    \,dr}|
\xi_X|  \,ds.
\end{equation}
Consequently, since
\[
\sup_{t \in[0, t_0]}\bigl|\widetilde{W}_1(t)-\widetilde{W}_2(t)\bigr|
\le\varepsilon, \qquad \bigl|H(0)\bigr|=|z_1-z_2| \le\nu
\varepsilon,
\]
recalling that $|f_1(t_0, z_1) - f_2(t_0,z_2)| = |H(t_0)|$, we get the estimate
%
%e11 #&#
%e12 #&#
\begin{eqnarray}\label{G}
&& \bigl|f_1(t_0, z_1) -
f_2(t_0,z_2)\bigr|
\nonumber\\[-8pt]\\[-8pt]
&&\qquad
\le\varepsilon \biggl( \nu e^{\int_0^{t_0}
\operatorname{Re}\psi_X(s)    \,ds} + \int_0^{t_0}
e^{\int_s^{t_0}
\operatorname{Re}\psi_X(r)
 \,dr}|\xi_X|  \,ds \biggr).\nonumber
\end{eqnarray}
The right-hand side in (\ref{G}) can be estimated in different ways
depending on what data is available. We would like an estimate that
depends only on $\varepsilon$ and $d=\operatorname{dist}(\{z_1,z_2 \}, \partial X)$.
Estimating naively, using only the fact that points flow away from
$\partial X$ under the reverse flow, gives a bound of order
$\varepsilon
e^{O(d^{-2})}$. (This kind of estimate was used in \cite{BJK}.) We
shall see that we can do much better.

%s2.3.1 #&#
\subsubsection{The chordal case}\label{sec2.3.1}
To give some intuition, let us first
briefly discuss the easier chordal case which will be treated in
greater detail in \cite{JVRW}. Assume $\nu=1 $ for simplicity. Write
$z_j(t) = h_j(t,z_j)-\widetilde{W}_j(t)$. We can apply the Cauchy--Schwarz
inequality to get
\begin{eqnarray*}
\int_0^{t} \operatorname{Re}
\psi_{\mathbb{H}} (t)  \,dt & \le&\int_0^{t}
\frac
{2}{|z_1(t)z_2(t)|}  \,dt
\\
& \le& \biggl( \int_0^{t} \frac{2}{|z_1(t)|^2}  \,dt
\biggr)^{1/2} \biggl( \int_0^{t}
\frac{2}{|z_2(t)|^2}  \,dt \biggr)^{1/2}.
\end{eqnarray*}
Since $\partial_t \log\operatorname{Im}z_j(t) = 2/|z_j(t)|^2$, this
can now be used
to show that the right-hand side of (\ref{G}) is bounded by
$\varepsilon
d^{-1}$ times a constant depending only on $T$, if $\operatorname
{Im}z_j(0) \ge d,
j=1,2$. (Note that there is no logarithmic correction.)

%re2 #&#
\begin{Remark*}
The estimate $\varepsilon d^{-1}$ is essentially sharp if no
further assumptions are made. Indeed, consider a driving term $W_1(t)$
generating a Loewner chain such that for some fixed $p<1$ very close to
$1$, $t_0 > 0$, there is a constant $c>0$ such that $|f_1'(t_0,W_1(t_0)
+ id)| \ge c d^{-p}$ as $d \to0$. (As shown in \cite{LMR}, one can
take $W_1(t)=\kappa\sqrt{t_0-t}$ with $\kappa$ very close to but
smaller than $4$. The curve traces a kind of logarithmic spiral.) If we
let $W_2(t)=W_1(t)+\varepsilon$, then
$f_2(t,z)=f_1(t,z-\varepsilon) + \varepsilon$. Hence, for
$\varepsilon\le d/2$, by Koebe's
distortion theorem,
\begin{eqnarray*}
&&\bigl|f_2\bigl(t_0,W_1(t_0)+id
\bigr)-f_1\bigl(t_0,W_1(t_0)+id
\bigr)\bigr|
\\
&&\qquad \ge\bigl|f_1\bigl(t_0,W_1(t_0)+id-
\varepsilon \bigr)-f_1\bigl(t_0,W_1(t_0)+id
\bigr)\bigr|-\varepsilon
\\
&&\qquad \ge c \varepsilon\bigl|f'_1\bigl(t_0,
W_1(t_0)+id\bigr)\bigr| \ge c \varepsilon
d^{-p}.
\end{eqnarray*}
A similar example can be constructed for the radial case.
\end{Remark*}

If more information is available, one can do better. The reader may
check that $\partial_t \operatorname{Re}\log h_j'(t,z) =
\operatorname{Re}(2/z_j(t)^2)$. From
this, one can see that the bound can be expressed in terms of the
derivatives $f_j'$. In fact, in joint work with Rohde and Wong, \cite
{JVRW}, we show that
\begin{eqnarray*}
&& \bigl\llvert f_1(t_0,z)-f_2(t_0,z)
\bigr\rrvert
\\
&&\qquad \le\varepsilon\exp \biggl\{ \frac
{1}{2} \biggl[ \log\frac{I_{t_0,y} \llvert f_1'(t_0,z)\rrvert }{y}
\log \frac{I_{t_0,y}\llvert f_2'(t_0,z)\rrvert }{y} \biggr]^{1/2} + \log \log \frac{I_{t_0,y}}{y}
\biggr\},
\end{eqnarray*}
where $I_{t,y}=\sqrt{4t+y^2}$. If a nontrivial power-law bound on the
growth of the derivative at time $t_0$ holds, that is, if $c_j < \infty
$ and $\beta_j<1$ are such that for $j=1,2$,
%
%e13 #&#
\begin{equation}
\label{der-est-t0} \bigl|f_j'\bigl(t_0,W_j(t_0)+id
\bigr)\bigr| \le c_j d^{-\beta_j}, \qquad d \le d_0,
\end{equation}
then one gets a bound in (\ref{G}) of order at most $c \varepsilon
d^{-(1/2)[(1+\beta_1)(1+\beta_2)]^{1/2}} \log d^{-1}$, where $c$ depends
only on $c_j, \beta_j$, $j=1,2$.

%s2.3.2 #&#
\subsubsection{The radial case}\label{sec2.3.2}
We now consider the radial setting
$X=\mathbb{D}$. In order to bound the right-hand side of (\ref{G}) we
need to estimate $\int_s^{t_0} \operatorname{Re}\psi_{\mathbb
{D}}(s)   \,ds$. The idea
is to prove that for a constant $q$ slightly larger than $1$,
\[
\operatorname{Re}\psi_\mathbb{D}(t)\le q \frac{\sqrt {1+|z_1(t)|}}{|1-z_1(t)|} \cdot
\frac
{\sqrt{1+|z_2(t)|}}{|1-z_2(t)|},
\]
where for $t \in[0,t_0]$, we define
\[
z_j(t)=h_j(t,z_j)\overline{
\widetilde{W}_j(t)}.
\]
Note that $|z_j(0)|=|z_j|$. Once we have this estimate, we can apply
the Cauchy--Schwarz inequality to the corresponding bound on $\int_s^{t_0} \operatorname{Re}\psi_{\mathbb{D}}(s)   \,ds$ to decouple the
two flows and
then compare with
%
%e14 #&#
\begin{equation}
\label{nov1211.1} \frac{1+|z_j(t)|}{|1-z_j(t)|^2}=\partial_t \log
\bigl(1-\bigl|z_j(t)\bigr| \bigr).
\end{equation}
This last identity follows from the reverse-time Loewner equation (\ref
{reverselocal}). This will give a bound in (\ref{G}) of order
$\varepsilon
d^{-q}$, where $q$ can be taken arbitrarily close to $1$. (Arguing as
in the chordal case only gives a rough bound of order $\varepsilon
d^{-4}$, but
we shall actually make use of this bound below.) This is essentially
optimal in this general setting as we saw above.

%
%pr2.3 #&#
\begin{Proposition}
For $j=1,2$, let $(f_j, W_j)$ be $\mathbb{D}$-Loewner pairs. For any
$\rho>1$, there exist $\varepsilon_0=\varepsilon_0(\rho) > 0$,
$d_0=d_0(\rho) > 0$,
and $c=c(\rho)<\infty$ such that the following holds. Let $T < \infty$
and suppose that
\[
\sup_{t \in[0,T]}\bigl|W_1(t) - W_2(t)\bigr| \le
\varepsilon,
\]
where $\varepsilon< \varepsilon_0$. Then for any $z_1, z_2 \in
\mathbb{D}$ with
$|z_1-z_2| \le\varepsilon$ and $|z_1|, |z_2| \le1-d$ with $(4
\varepsilon)^{1/\rho}
\le d \le d_0$,
%
%e15 #&#
\begin{equation}
\label{diff-estimate} \bigl\llvert f_1(T, z_1) -
f_2(T,z_2) \bigr\rrvert \le c \varepsilon
d^{-\rho}.
\end{equation}
\end{Proposition}

\begin{pf}
By factoring out $\widetilde{W}_1 \widetilde{W}_2$, we can write
%
%e16 #&#
\begin{eqnarray}
\label{dec28.1}
&& \operatorname{Re}\psi_{\mathbb{D}}(t)\nonumber
\\
&&\qquad  =\operatorname{Re}\biggl( \frac
{z_1(t)z_2(t)-1-(z_1(t)+z_2(t))+O(\varepsilon)}{(1-z_1(t))(1-z_2(t))} \biggr)
\\
&&\qquad  =\frac{\operatorname{Re} \{  (
z_1(t)z_2(t)-1-(z_1(t)+z_2(t))+O(\varepsilon)
) (1-\overline{z_1(t)})(1-\overline{z_2(t)})  \}
}{|1-z_1(t)|^2|1-z_2(t)|^2}.\hspace*{-17pt}\nonumber
\end{eqnarray}
This uses that $\overline{\widetilde{W}_1(t)}\widetilde{W}_2(t)= 1+
O(\varepsilon)$ in
the sense that $|\overline{\widetilde{W}_1(t)}\widetilde{W}_2(t)-1| \le c
\varepsilon$
for a universal constant $c$.
For $z, w \in\overline{\mathbb{D}}$ we now consider the function
\[
R(z,w)=\frac{\operatorname{Re} \{  ( zw-1-(z+w) )
(1-\overline
{z})(1-\overline{w})  \}}{|1-z||1-w|\sqrt{(1+|z|)(1+|w|)}},
\]
which is bounded and continuous on the closed bi-disk $\mathbb
{D}\times\mathbb{D}$.
We claim that $\sup_{z,w \in\partial\mathbb{D}}R(z,w) \le1$. A computation
shows that $R$ simplifies when $|z|=|w|=1$ so that
\[
R(z, w) =\frac{(1-\operatorname{Re}z)(1- \operatorname{Re}w) +
\operatorname{Im}z \operatorname{Im}w}{2\sqrt{ (1-\operatorname{Re}z)(1-
\operatorname{Re}w) }} \qquad \bigl(|z| = |w|=1\bigr).
\]
By changing coordinates $z=e^{i\theta}$ and $w=e^{i\mu}$, with
$\theta,
\mu\in[0, 2\pi]$, in the last expression we find
\[
\bigl( R\bigl(e^{i\theta}, e^{i\mu}\bigr) \bigr)^2=
\cos^2 \biggl( \frac
{\theta-\mu
}{2} \biggr) \le1.
\]
Let $\delta>0$ be such that $\rho=1+2\delta$; we assume that $\delta$
is small. By the last expression and the continuity of $R$, there
exists $\varepsilon'(\delta)>0$ such that if $1-\varepsilon' \le
|z|,|w| \le1$ then
$R(z,w) \le1+\delta/2$. We will fix $\varepsilon'$ from now on. We
can think
of $\varepsilon'$ as small but macroscopic compared to $\varepsilon$.
Returning to the
flows, by (\ref{dec28.1}) and the bound on $R$, if $\varepsilon$ is
sufficiently small compared to $\delta$, we have the estimate
%
%e17 #&#
\begin{eqnarray}
\label{fine} \operatorname{Re}\psi_\mathbb{D}(t) & =& \operatorname{Re}
\biggl( \frac{z_1(t)z_2(t)-1-(z_1(t)+z_2(t))+O(\varepsilon
)}{(1-z_1(t))(1-z_2(t))} \biggr)
\nonumber\\[-8pt]\\[-8pt]
& \le&(1+\delta) \frac{\sqrt{1+|z_1(t)|}}{|1-z_1(t)|} \cdot\frac
{\sqrt {1+|z_2(t)|}}{|1-z_2(t)|}, \qquad 0 \le t \le
\tau,\nonumber
\end{eqnarray}
where
\[
\tau=\inf\bigl\{t \ge0\dvtx  \min\bigl\{\bigl|z_1(t)\bigr|, \bigl|z_2(t)\bigr|
\bigr\} \le1-\varepsilon'\bigr\}.
\]
We will assume that $\tau> 0$ as there is nothing to prove otherwise.
We split the integral
\[
\int_0^{T}\operatorname{Re}
\psi_{\mathbb{D}}(s)  \,ds =\int_0^{\tau}
\operatorname{Re}\psi _{\mathbb{D}}(s)  \,ds + \int_{\tau}^{T}
\operatorname{Re}\psi _{\mathbb{D}}(s)  \,ds.
\]
We estimate the first integral using (\ref{fine}) and the
Cauchy--Schwarz inequality. We get, for $0 \le s \le\tau$:
\[
\int_s^{\tau}\operatorname{Re}
\psi_{\mathbb{D}}(s)  \,ds \le (1+\delta) \biggl( \int_0^{\tau}
\frac{1+|z_1(s)|}{|1-z_1(s)|^2}  \,ds \biggr)^{1/2} \biggl(\int_0^{\tau}
\frac{1+|z_2(s)|}{|1-z_2(s)|^2}  \,ds \biggr)^{1/2}.
\]
Using (\ref{nov1211.1}), we see that for $0 \le s\le\tau$,
%
%e18 #&#
\begin{equation}
\label{dec28.2} \int_s^{\tau}\operatorname{Re}
\psi_{\mathbb{D}}(s)  \,ds \le (1+\delta) \biggl( \log \biggl( \frac{\varepsilon'}{1-|z_1|}
\biggr) \biggr)^{1/2} \biggl( \log \biggl( \frac
{\varepsilon'}{1-|z_2|} \biggr)
\biggr)^{1/2}.
\end{equation}
Thus, with $\max\{|z_1|, |z_2| \} = 1-d$ we conclude that
%
%e19 #&#
\begin{eqnarray}
\label{ztau} \bigl|z_1(\tau)-z_2(\tau)\bigr| & \le&\varepsilon
\biggl(e^{\int_0^{\tau}
\operatorname{Re}\psi_\mathbb{D}(s)
 \,ds} + \int_0^{\tau}
e^{\int_s^{\tau} \operatorname{Re}\psi
_\mathbb{D}(r)    \,dr}|\xi_\mathbb{D}|  \,ds \biggr)
\nonumber
\\
& \le&\varepsilon \biggl(\frac{\varepsilon'}{d} \biggr)^{1+\delta
} \biggl(1+\log
\frac{\varepsilon
'}{d} \biggr)
\\
& \le&2 \varepsilon \biggl(\frac{\varepsilon'}{d} \biggr)^{1+\delta
}\log
\frac{1}{d},\nonumber
\end{eqnarray}
if $d \le1/e$.
Here, we also used that
\[
\bigl|\xi_{\mathbb{D}}(s)\bigr|\le\frac{\sqrt{1+|z_1(s)|}}{|1-z_1(s)|}\cdot \frac{\sqrt {1+|z_2(s)|}}{|1-z_2(s)|},
\]
the integral of which is estimated using the Cauchy--Schwarz inequality
as above.
Recall that $1+2\delta= \rho$. There is a $d_0(\rho)> 0 $ such that $d
\le d_0$ implies that $d^\rho=d^{1+2 \delta} \le d^{1+\delta}/\log
(1/d)$. Consequently, if $\varepsilon$ is sufficiently small we can
choose $d$
such that
\[
4 \varepsilon\bigl(\varepsilon'\bigr)^\delta\le4
\varepsilon\le d^{1+2\delta
} \le d_0^{1+2\delta}
\]
and then use (\ref{ztau}) to get the estimate
%
%e20 #&#
\begin{eqnarray}
\label{ee'} \max\bigl\{ \bigl|z_1(\tau)\bigr|, \bigl|z_2(\tau)\bigr|
\bigr\} & \le&1-\varepsilon' + \bigl|z_1(\tau
)-z_2(\tau)\bigr|
\nonumber
\\
& \le&1-\varepsilon'+ 2 \varepsilon\bigl(\varepsilon'
\bigr)^{1+\delta
}{d}^{-(1+2\delta)}
\\
& \le&1-\frac{\varepsilon'}{2}.\nonumber
\end{eqnarray}
Note the easy bound
%
%e21 #&#
\begin{equation}
\label{alw} \operatorname{Re}\psi_\mathbb{D}(t) \le\bigl|
\psi_\mathbb{D}(t)\bigr| \le4 \frac{\sqrt {1+|z_1(t)|}}{|1-z_1(t)|}\cdot\frac{\sqrt{1+|z_2(t)|}}{|1-z_2(t)|}, \qquad 0
\le t \le T.
\end{equation}
Combining this with the Cauchy--Schwarz inequality, (\ref{nov1211.1})
and (\ref{ee'}) gives
\[
\int_\tau^{T} \operatorname{Re}
\psi_\mathbb{D}(s)  \,ds \le4 \log \frac{2}{\varepsilon'}.
\]
Putting things together, we get
\begin{eqnarray*}
\bigl|f_1(T,z_1) - f_2(T,z_1)\bigr|
& \le&\varepsilon \biggl(e^{\int_0^{T}
\operatorname{Re}\psi_\mathbb{D}(s)
 \,ds} + \int_0^{T}
e^{\int_s^{T} \operatorname{Re}\psi_\mathbb
{D}(r)    \,dr}|\xi_\mathbb{D}|  \,ds \biggr)
\\
& \le&2 \varepsilon\log\frac{1}{d} \exp \biggl\{(1+\delta) \log
\frac{\varepsilon
'}{d} + 4\log\frac{2}{\varepsilon'} \biggr\}
\\
& \le& c \varepsilon d^{-(1+2\delta)},
\end{eqnarray*}
where $c=c(\rho) < \infty$.
\end{pf}

%
%re3 #&#
\begin{Remark*}
We believe that the function $R(z,w)$ used in the last proof is bounded
by $1$ on the whole bi-disk, and with some work one should be able to
verify this. [However, this is not true for $|R(z,w)|$.] This would
allow for taking $\rho=1$ in (\ref{diff-estimate}). This would not
improve the resulting convergence rate in Theorem~\ref{lerw-thm}, so we
will not pursue this here. However, we do expect a bound of type
$\varepsilon
d^{-(1/2)[(1+\beta_1)(1+\beta_2)]^{1/2}} \log d^{-1}$ to hold in
the radial case, too. Having this estimate could slightly improve the
resulting convergence rate in Theorem~\ref{lerw-thm}.
\end{Remark*}

Suppose now that for $j=1,2$, $f_j$ satisfies the derivative estimate
(\ref{der-est-t0}) with $\beta=\beta_j$ and $c=c_j$. [In the radial
case, we consider the radial version of (\ref{der-est-t0}) and take
$\beta_j=1$; indeed, it is a general fact about (normalized) conformal
maps that~(\ref{der-est-t0}) always holds with $\beta=1$ for some
constant universal constant $c<\infty$.] Set
%
%e22 #&#
\begin{equation}
\label{q0} \rho_0=\rho_0(\beta_1,
\beta_2) = \cases{ 1, &\quad if $X=\mathbb{D}$;
\vspace*{3pt}\cr
\displaystyle
\tfrac{1}{2}\sqrt{(1+\beta_1) (1+\beta_2)}, &\quad
if $X=\mathbb{H}$.}
\end{equation}
Suppose $\rho> \rho_0$ and $p \in(0, 1/\rho)$. Let $\varepsilon>0$
and define
%
%e23 #&#
\begin{equation}
d_*=\varepsilon^p.
\end{equation}
We have proved that for any $z$ and $w$ with $|z-w| \le\varepsilon$
at distance at least $d_*$ from the boundary, if the driving terms
satisfy $\sup|W_1(t) - W_2(t)| \le\varepsilon$, then there are
$c=c(\rho, p)
< \infty$ and $\varepsilon_0=\varepsilon_0(\rho)>0$ such that if
$\varepsilon< \varepsilon_0$, then
\[
\bigl|f_1(t_0,z)-f_2(t_0,w)\bigr|
\le c \varepsilon^{1-\rho p}.
\]
By estimating using Cauchy's integral formula, we also get a bound
relating the derivatives: write $f_j(z)=f_j(z,t_0)$. Then with $d =
\operatorname{dist}(z, \partial X)$,
\[
\bigl|f'_1(z)-f'_2(z)\bigr|=
\frac{1}{2\pi} \biggl\llvert \oint_{|\zeta-z|=r}\frac
{f_1(\zeta)-f_2(\zeta)}{(z-\zeta)^2}  \,d\zeta
\biggr\rrvert \le c \varepsilon d^{-\rho} r^{-1},
\]
where $r \le d/2$. Taking $d=2r=\varepsilon^p$ this estimate combined
with the
reverse triangle inequality shows that there is a constant $c=c(\rho,
p, T) <\infty$ (recall that $t_0 \le T$) such that
\[
\sup_{z\dvtx  \operatorname{dist}(z, \partial X) \ge\varepsilon^p}\bigl\llvert \bigl|f'_1(z)\bigr|-\bigl|f'_2(z)\bigr|
\bigr\rrvert \le c \varepsilon^{1-(1+\rho)p}.
\]
We have proved the radial part of the following result. (The chordal
case is joint work with Rohde and Wong; see \cite{JVRW} for its
complete proof.)

%
%pr2.4 #&#
\begin{Proposition}\label{rough}
Let $X \in\{\mathbb{D}, \mathbb{H} \}$ and $T > 0$. Let $(f_j, W_j),
j=1,2$, be \mbox{$X$-}Loewner pairs so that $f_j$ solve (\ref{crpde}) with $W_j$
as driving terms and assume that the $f_j$ satisfiy (\ref{der-eq}) with
$\beta=\beta_j$ and $c=c_j < \infty$. Suppose $\rho> \rho_0$, where
$\rho_0$ is defined by (\ref{q0}). Assume that $z,w \in X$ and for
$\varepsilon>0$
\[
\sup_{t \in[0,T]}\bigl|W_1(t)-W_2(t)\bigr| \le
\varepsilon, \qquad |z-w| \le \varepsilon
\]
and for $p \in(0,1/\rho)$ define
%
%e24 #&#
\begin{equation}
d_*=\varepsilon^p.
\end{equation}
There exist $c=c(T, \rho, p, c_1, c_2)<\infty, \varepsilon
_0=\varepsilon_0(\rho, p)>0,
d_0=d_0(\rho)>0$ such that if
\[
d_* \le\operatorname{dist}\bigl(\{z,w\}, \partial X\bigr) \le d_0
\]
and $\varepsilon< \varepsilon_0$, then
\[
\sup_{t \in[0,T]}\bigl|f_1(t,z) - f_2(t,w)\bigr| +
\sup_{t \in[0,T]}\bigl| d_*\bigl|f'_1(t,z)\bigr| -
d_*\bigl|f'_2(t,z)\bigr| \bigr| \le c \varepsilon^{1-\rho p}.
\]
\end{Proposition}

One way to interpret the last proposition is that information about the
derivative of one of the conformal maps transfers to the other via the
Loewner equation if they are evaluated sufficiently far away from the
boundary. The proper scale (or resolution) is determined by the
distance between the driving terms. Note that we make no assumptions
about the regularity of the driving terms; the above results are
consequences of the structure of the Loewner equation alone.

%s2.4 #&#
\subsection{Supremum distance between Loewner curves}\label{dist-sect}
We will now consider two Loewner curves, $\gamma_j\dvtx  [0,T] \to X,
j=1,2$, generating the $X$-Loewner pairs $(f_j,W_j)$ and suppose that
%
%e25 #&#
\begin{equation}
\label{10313} \sup_{t \in[0,T]} \bigl|W_1(t)-W_2(t)\bigr|
\le\varepsilon.
\end{equation}
We are interested in estimating the supremum distance $\sup_{t \in
[0,T]}|\gamma_1(t)-\gamma_2(t)|$ when the curves are parameterized by
capacity, in terms $\varepsilon$. We have the following estimate.

%
%pr2.5 #&#
\begin{Proposition}\label{1031}
Let $X \in\{\mathbb{D}, \mathbb{H}\}$. For $j=1,2$, let $(f_j, W_j)$
be $X$-Loewner pairs generated by the curves $\gamma_j$ and suppose
that there are $d_0 > 0$ and $\beta_j, c_j$ such that $f_j$ satisfy
(\ref{der-eq}) with $\beta=\beta_j$ and $c=c_j$. Let $\rho> \rho_0$,
where $\rho_0$ is given by (\ref{q0}). Suppose that $\varepsilon>0$
is such that
\[
\sup_{t \in[0,T]} \bigl|W_1(t) -W_2(t)\bigr| \le
\varepsilon.
\]
Let $p \in(1, 1/\rho)$ and set $d = \varepsilon^p$. There exist
$c=c(T, \rho,p) < \infty$ and $\varepsilon_0=\varepsilon_0(\rho,p) > 0$ such
that if $\varepsilon<\varepsilon_0$, then
%
%e26 #&#
%e27 #&#
%e28 #&#
\begin{eqnarray}
\label{10312}
&& \sup_{t \in[0,T]} \bigl|\gamma_1(t) -\gamma_2(t)\bigr|\nonumber
\\
&&\qquad \le c\varepsilon ^{1-\rho p} + c \sup_{t \in[0,T]} \bigl(\bigl|
\gamma_1(t)-f_1\bigl(t, (1-d)W_1(t)
\bigr)\bigr|
\\
&&\hspace*{114pt} {}
+\bigl|\gamma _2(t)-f_2\bigl(t, (1-d)W_2(t)
\bigr)\bigr|\bigr)\nonumber
\end{eqnarray}
with $f_j(t, (1-d)W_j(t))$ replaced by $f_j(t, W_j(t) + id)$ in the
chordal case.
\end{Proposition}

\begin{pf}
We will do the radial case.
Write
\begin{eqnarray*}
\bigl|\gamma_1(t)-\gamma_2(t)\bigr| &\le& \bigl|\gamma
_1(t)-f_1\bigl(t,(1-d)W_1(t)\bigr)\bigr|
\\
&&{} + \bigl|f_1\bigl(t,(1-d)W_1(t)\bigr)-f_1
\bigl(t,(1-d)W_2(t)\bigr)\bigr|
\\
&&{} + \bigl|f_1\bigl(t,(1-d)W_2(t)\bigr) - f_2
\bigl(t,(1-d)W_2(t)\bigr)\bigr|
\\
&&{} + \bigl|f_2\bigl(t,(1-d)W_2(t)\bigr)-\gamma_2(t)\bigr|.
\end{eqnarray*}
Denote by $b_1, \ldots, b_4$ the four terms on the right-hand side in
the last inequality in the order in which they appear. By the
distortion theorem, since $d \ge\varepsilon$ we have that
\[
b_2 \le c \operatorname{dist}\bigl(f_1
\bigl(t,(1-d)W_1(t)\bigr), \partial f_1(t, \mathbb{D})
\bigr) \le c b_1.
\]
Finally, by Proposition~\ref{rough}, $b_3 \le c\varepsilon^{1-\rho p}$.
\end{pf}

%
%co2.6 #&#
\begin{Corollary}\label{lip}
For $j=1,2$, let $(f_j, W_j)$ be $\mathbb{H}$-Loewner pairs generated
by the curves $\gamma_j$ and assume that (\ref{10313}) holds. Suppose
that there exist $d_0>0$, $c < \infty$, and $\beta<1$ such that the
$f_j$ satisfy the estimate (\ref{der-eq}). Then for every
\[
r < 2\frac{1-\beta}{3-\beta},
\]
there exist $c=c(r,T) < \infty$ and $\varepsilon_0=\varepsilon
_0(r,T)>0$ such that if
$\varepsilon<\varepsilon_0$, then
\[
\sup_{t \in[0,T]}\bigl|\gamma_1(t) - \gamma_2(t)\bigr|
\le c \varepsilon^{r}.
\]
\end{Corollary}

\begin{pf}
Under our assumptions $\rho_0=(1+\beta)/2$. Let $\rho> \rho_0$ and $0
< p < 1/\rho$. We set $d = \varepsilon^{p}$, apply Proposition~\ref
{1031}, and
integrate the bound on the derivatives to see that for $\varepsilon>0$
sufficiently small,
\[
\sup_{t \in[0,T]}\bigl|\gamma_1(t)-\gamma_2(t)\bigr|
\le c \bigl( \varepsilon^{1-\rho p} + \varepsilon^{p(1-\beta)} \bigr).
\]
We optimize over exponents to find the stated bound for $r$.
\end{pf}

The proof of the next corollary is an analog for Loewner curves of the
well-known fact that the Riemann map onto a H\"older domain satisfies a
power-law bound on the growth of the derivative.

%co2.7 #&#
\begin{Corollary}\label{holder2}
For $j=1,2$, let $(f_j, W_j)$ be $\mathbb{H}$-Loewner pairs generated
by the curves $\gamma_j$ and assume that (\ref{10313}) holds. Suppose
that both curves are \mbox{H\"older-}$\alpha$ continuous in the capacity
parameterization, where $\alpha> 0$. Then for every
\[
r < \frac{2\alpha}{1+\alpha},
\]
there exist $c=c(r,T) < \infty$ and $\varepsilon_0=\varepsilon
_0(r,T)>0$ such that if
$\varepsilon<\varepsilon_0$, then
\[
\sup_{t \in[0,T]}\bigl|\gamma_1(t) - \gamma_2(t)\bigr|
\le c \varepsilon^{r}.
\]
\end{Corollary}

\begin{pf}
We will prove a bound on the growth of the derivative and then apply
the previous corollary. It is enough to consider $f(t,z):=f_1(t,z)$
since we made the same assumptions on both Loewner chains. Write
$\gamma
=\gamma_1$ and $W=W_1$ and for $t,t+s \in[0,T]$, let
\[
\tilde{\gamma} = f^{-1}\bigl(t, \gamma[t,t+s]\bigr).\vadjust{\goodbreak}
\]
Then $\tilde\gamma$ is a curve in $\mathbb{H}$ ``rooted'' at $W(t)$. Set
$d=\operatorname{diam}\tilde{\gamma}$. Let $z \in\tilde\gamma$ be
a point such that
$|z - W(t)|=d/2$ and let $\Gamma$ be the hyperbolic geodesic in
$\mathbb{H}$
connecting $W(t)$ with $z$. Then $\Gamma$ contains a point $w$ with
$\operatorname{Im}w \ge d/4$. Note that by the distortion theorem,
$|f'(t,w)| \asymp
|f'(t,W(t)+id)|$ so that Koebe's $1/4$ theorem implies that there is a
universal constant $c>0$ such that
\[
\mathcal{B} \bigl(f(t,w), cd\bigl|f'\bigl(t, W(t)+id\bigr)\bigr| \bigr)
\subset f \bigl(t, \mathcal{B}(w, d/8) \bigr).
\]
[Here, and in the sequel $\mathcal{B}(z,r) = \{w\dvtx  |w-z| < r\}$.] Consequently,
%
%e29 #&#
\begin{equation}
\label{hol1} \operatorname{diam}f(t, \Gamma) \ge c d\bigl|f'\bigl(t,
W(t)+id\bigr)\bigr|.
\end{equation}
On the other hand, by the Gehring--Hayman theorem (see Chapter~4 of
\cite{Pom92}) and the assumption on $\gamma$, we have that there are
constants $c, c' < \infty$, depending only on the constant in the
modulus of continuity for $\gamma$, such that
\[
\operatorname{diam}f(t, \Gamma) \le c \operatorname{diam}\gamma [t,t+s] \le
c' s^\alpha.
\]
Hence, using (\ref{hol1}), there is a constant $c < \infty$ such that
\[
d\bigl|f'\bigl(t,W(t) + id\bigr)\bigr| \le c s^{\alpha} \le
c' d^{2\alpha},
\]
where the last inequality follows since $\operatorname{hcap}\tilde
\gamma= 2s$ so
that there is a universal constant $c<\infty$ such that $s \le c d^2$.
The diameter $d$ depended on $s$, but every $d$ sufficiently small can
be written like this since $s \mapsto d$ is an increasing continuous function.
\end{pf}

%
%re4 #&#
\begin{Remark*}
If $\gamma(t)$ is H\"older-$\alpha$ continuous in the capacity
parameterization, then its driving term is at least H\"older-$\alpha
/2$: using the notion of the proof of Corollary~\ref{holder2}, we note
that by the Beurling estimate, $\operatorname{diam}\tilde\gamma\le
c  s^{\alpha
/2}$ and by Lemma~2.1 of \cite{LSW}, we have $|W(t+s) - W(t)| \le c
\operatorname{diam}\tilde\gamma\le c'  s^{\alpha/2}$.
\end{Remark*}

%
%s3 #&#
\section{Geometric conditions}\label{geom}
This section develops a geometric condition that we will use in place
of a bound on the growth of the derivative of the conformal map in
order to measure the regularity of a Loewner curve locally at the tip.
As pointed out in the \hyperref[sec1.1]{Introduction}, several similar conditions have
appeared in the literature. We will work in the radial setting, but the
results hold also in the chordal setting with minor modifications in
their statements and proofs.

Let $D \ni0$ be a simply connected domain. Let $\psi\dvtx  D \to\mathbb
{D}$ be the uniformizing conformal map. We consider a radial Loewner
curve $\gamma\dvtx  [0, T] \to D$, that is, the conformal image of $\gamma$
in $\mathbb{D}$ using the conformal map $\psi$ is a $\mathbb{D}$-Loewner
curve. In this section we write $D_t$ for the connected component of $D
\setminus\gamma[0,t]$ containing the origin.

%s3.1 #&#
\subsection{Tip structure modulus}\label{sec3.1}
For $s,t \in[0,T]$ with $s \le t$, we let $\gamma_{s,t}$ denote the
curve determined by $\gamma(r), r \in[s,t]$. For a crosscut $\mathcal
{C}$ of
$D_t$, we write $J_{\mathcal{C}}$ for the component of $D_t \setminus
\mathcal{C}$ of
smaller diameter.\vadjust{\goodbreak}

For each $0 \le t \le T$ and $\delta> 0$, let $S_{t,\delta}$ be the
collection of crosscuts of $D_t$ of diameter at most $\delta$ that
separate $\gamma(t)$ from $0$ in $D_t$. For a crosscut $\mathcal{C}
\in
S_{t,\delta}$, define
\[
s_{\mathcal{C}}=\inf\bigl\{s > 0\dvtx  \gamma[t-s,t] \cap\overline{\mathcal {C}}
\neq \varnothing\bigr\}, \qquad\gamma_{\mathcal{C}} = \bigl(\gamma(r), r \in
[t-s_{\mathcal{C}},t]\bigr).
\]
(We set $s_{\mathcal{C}} = t$ if $\gamma$ never intersects $\overline
{\mathcal{C}}$.)
For $\delta> 0$, we define the \textit{tip structure modulus} of
$(\gamma(t),   t \in[0, T])$ in $D$, written $\eta_{\mathrm{tip}}(\delta)$,
to be the maximum of $\delta$ and\looseness=-1
%
%e30 #&#
\begin{equation}
\label{LSM} \sup_{t \in[0, T]} \sup_{\mathcal{C} \in S_{t, \delta}}
\operatorname{diam}\gamma _{\mathcal{C}}.
\end{equation}

%re5 #&#
\begin{Remark*}
In the chordal setting, we consider instead crosscuts separating
$\gamma
(t)$ from $\infty$ in $H_t$ in the definition of the structure modulus.
The remaining construction is the same.
\end{Remark*}

It is useful to introduce some more terminology. Given $0<\delta\le
\eta$, we will say that the curve $\gamma$ has a $(\delta, \eta
)$-\textit{bottleneck} in $D$ if there exist $t \in[0,T]$ and $\zeta
\in\partial D_t$ such that $\gamma(t)$ and $\zeta$ can be connected by
a crosscut $\mathcal{C}_t$ of $D_t$ and $\operatorname
{diam}J_{\mathcal{C}_t} \ge\eta$
while $\operatorname{diam}\mathcal{C}_t \le\delta$. This definition
is similar to
the one for ``quasi-loops'' given by Schramm in \cite{SchrammIJM}. We
say that the bottleneck is at $z_0$ if the points $\zeta$ and $\gamma
(t)$ in the previous definition are contained in the disk $\mathcal{B}(z_0,
\eta/4)$.

Similarly, given $0 < \delta\le\eta$ we will say that the curve
$\gamma$ has a \textit{nested} \mbox{$(\delta, \eta)$-}\textit{bottleneck} in
$D$ if there exist $t \in[0,T]$ and $\mathcal{C} \in S_{t,\delta}$ with
\[
\operatorname{diam}\gamma_{\mathcal{C}} \ge\eta.
\]
That $\gamma(t),   t \in[0,T]$ has no nested $(\delta, \eta
)$-bottleneck in $D$ is clearly equivalent to having the inequality
$\eta_{\mathrm{tip}}(\delta) \le\eta$.

%
%re6 #&#
\begin{Remark*}
The definition of nested bottleneck is independent of the particular
chosen parameterization of the curve in the sense that any increasing
reparameterization would do in the definition. The definition is not,
however, symmetric with respect to reversibility of the curve.
\end{Remark*}

%
%f1 #&#
\begin{figure}[t]

\includegraphics{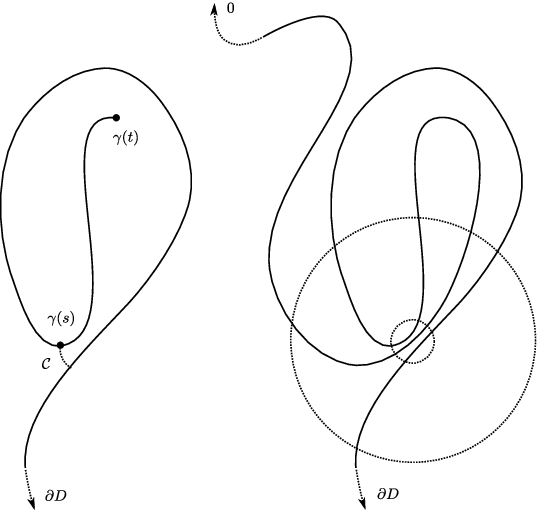}

\caption{A nested $(\delta,\eta)$-bottleneck with $\operatorname
{diam}\mathcal{C} = \delta$ and $\operatorname{diam}\gamma_\mathcal{C} \ge\eta$,
where $\gamma_\mathcal{C}=\gamma[s,t]$. A~$6$-crossing event of a $(\delta,\eta)$-annulus for
the whole curve.}\label{fig1}
\end{figure}
The term ``structure modulus'' is borrowed from Warschawski \cite{W}
who used it in the following sense: the ``structure modulus of the
boundary of $D$'' is defined by the function
\[
\eta_W(\delta)=\sup_{\mathcal{C}} \operatorname{diam}J_\mathcal{C},
\]
where the supremum is over all crosscuts (of $D$) of diameter at most
$\delta$ and $J_\mathcal{C} \subset\partial D$ is the subarc of
smaller diameter separated from $0$ by $\mathcal{C}$. Intuitively,\break the
decay rate of $\eta_W$ places a restriction on
bottlenecks/outward-pointing cusps in the boundary and this gives
estimates on the regularity of the Riemann mapping from $\mathbb{D}$.
For example, $D$ is a John domain if and only if $\eta_W(\delta)\le A
\delta$ for some constant \mbox{$A <\infty$}. One can use this to show (see
\cite{W}) that if $h<2/(A^2 \pi^2)$, then the Riemann map from
$\mathbb
{D}$ is H\"older-$h$ on the closed unit disk. The tip structure modulus
is the natural analogue to $\eta_W$ for Loewner curves; see
Theorem~\ref{geom-holder} below. Moreover, and importantly, the tip structure
modulus is related to annuli crossing events (see~Figure~\ref{fig1}),
the probabilities of which are often known how to control for
discrete-model curves; the connection between annuli crossings and
regularity of curves is well known; see, for example, \cite{AB}.

%s3.2 #&#
\subsection{Distance to the tip}\label{sec3.2}
Let $(f,W)$ be a $\mathbb{D}$-Loewner pair and assume it is generated
by a curve $\gamma$. We use the notation
\[
\Delta_{t}(d) = \operatorname{dist}\bigl(f_t
\bigl((1-d)W_t\bigr), D_t\bigr),
\]
where $W_t=e^{i\xi_t}$ is the driving term for $(f_t)$. Note that
Koebe's distortion theorem implies that
\[
\Delta_{t}(d) \asymp d\bigl|f'_t
\bigl((1-d)W_t\bigr)\bigr|.
\]
Recall also that for each $t$, the tip of the curve is given by taking
the radial limit
\[
\gamma(t)=\lim_{d \to0+}f_t\bigl((1-d)W_t
\bigr).
\]
We saw in Section~\ref{dist-sect} that we need to obtain uniform (in
$t$) bounds on
\[
\bigl|\gamma(t) - f_t\bigl((1-d)W_t\bigr)\bigr|.
\]
A lower bound on this quantity is clearly given by $\Delta_t(d)$ and if
we have a bound for $\eta_{\mathrm{tip}}(\delta)$ in terms of
$\delta$, then we
can also give an estimate from above in terms of $\Delta_t(d)$. We need
the following lemma. (See Figure~\ref{fig2} for a sketch illustrating
the proof.)

%f2 #&#
\begin{figure}[t]

\includegraphics{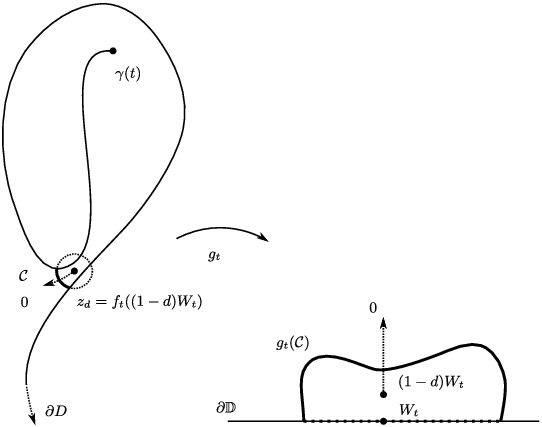}

\caption{Sketch for the proof of Lemma~\protect\ref{1027}. The crosscut
$g_t(\mathcal{C})$ separates $(1-d)W_t$ and $g_t(\mathcal{E}) \subset
{\partial }\mathbb{D}$ from $0$ in $\mathbb{D}$. The harmonic
measure of $g_t(\mathcal{E})$
from $(1-d)W_t$ is at least $1/2$. Hence, $W_t \in g_t(\mathcal{E})$.}\label{fig2}
\end{figure}

%le3.1 #&#
\begin{Lemma}\label{1027}
Let $T <\infty$ be given. There exist constants $0< \rho_1, c_1<
\infty
$ with $\rho_1$ universal and $c_1=c_1(T)$ such that the following
holds. Let $\gamma$ be a curve in $\mathbb{D}$ generated by the
Loewner pair
$(f, W)$. Let $t \in[0,T]$. If $\Delta_t(d) < c_1$ then there is a
crosscut $\mathcal{C}=\mathcal{C}_t$ of $D_t$ that separates $f_t((1-d)W_t)$
and $\gamma(t)$ from $0$ in $D_t$ while
\[
\operatorname{diam}\mathcal{C}\le\rho_1 \Delta_t(d).
\]
Moreover, $\mathcal{C}$ can be taken to be a subarc of $\mathcal
{B}
(f_t((1-d)W_t), \rho_1   \Delta_t(d)/2 )$.
\end{Lemma}

\begin{pf}
Let $t \in[0,T]$ and set
\[
z_d=f_t\bigl((1-d)W_t\bigr).\vadjust{\goodbreak}
\]
We will write
\[
\Delta=\Delta_{t}(d)=\operatorname{dist}(z_d, \partial
D_t).
\]
For $\rho> 1$, consider $(\partial\mathcal{B}(z_d, \rho  \Delta))
\cap
D_t$. The components of this set form crosscuts of $D_t$ and we let
$C_0$ be the subset of those crosscuts that separate $z_d$ from $0$ in
$D_t$. (Since the inner radius of $D_t$ from $0$ is bounded below by
$e^{-T}/4$, $C_0$ is nonempty whenever $\rho\Delta$ is smaller than,
say, $e^{-T}/16$.) Let $\mathcal{C}_{\rho}$ be the unique crosscut in $C_0$
with the property that it separates every other member in $C_0$ from
$0$ in $D_t$. Let $\mathcal{O}_{\rho}$ be the component of $D_t
\setminus\mathcal{C}_{\rho}$ that contains $z_d$ and let $\mathcal
{E}_{\rho}=
\partial\mathcal{O}_{\rho} \setminus\mathcal{C}_{\rho}$. By Beurling's
projection theorem and the maximum principle, there exists a universal
$\rho_0 <\infty$ and for each $\rho> \rho_0$ a constant
$c_0=c_0(\rho,
T)>0$ such that if $\Delta< c_0$ then we have the following lower bound
on harmonic measure:
%
%e31 #&#
\begin{equation}
\label{oct25.1} \omega(z_d, \mathcal{E}_{\rho},
\mathcal{O}_{\rho}) > 1/2.
\end{equation}
Let $\mathcal{O}:=\mathcal{O}_{2\rho_0}$, $\mathcal{C}:=\mathcal
{C}_{2\rho_0}$ and
$\mathcal{E}:=\mathcal{E}_{2\rho_0}$. Let $c_1=c_1(T) < \infty$ be such
that if $\Delta< c_1$, then the diameter of the preimage of $\mathcal
{C}$ in
$\mathbb{D}$ is at most $1/2$ and (\ref{oct25.1}) holds with $\rho$ replaced
by $2\rho_0$. (Existence of such a $c_1$ follows from Beurling's
projection theorem.) We shall assume that $\Delta< c_1$ in the sequel.
We claim that the preimage of $\mathcal{E}$ in $\partial{\mathbb{D}}$
is an arc containing the point $W_t$. Indeed, it is clear that it is an
arc of $\partial\mathbb{D}$. If $g_t=f_t^{-1}$ then $g_t(\mathcal{C})$
is a crosscut of $\mathbb{D}$ separating $g_t(\mathcal{E})$ and
$(1-d)W_t$ from $0$. By conformal invariance, the maximum principle and
(\ref{oct25.1}), the harmonic measure of $g_t(\mathcal{E})$ from
$(1-d)W_t$ is strictly bigger than $1/2$. Write $W_t=e^{i\xi_t}$. Note
that by symmetry, the harmonic measure from $(1-d)W_t$ of $\{e^{i(\xi_t
+ \theta)}\dvtx  0 \le\theta\le\pi\}$ in $\mathbb{D}$ is exactly $1/2$.
Therefore, if $W_t=e^{i\xi_t} \notin g_t(\mathcal{E})$, then the arc
$g_t(\mathcal{E})$ must contain the point $e^{i(\xi_t + \pi)}$. Since
$g_t(\mathcal{C})$ separates $(1-d)W_t$ and $e^{i(\xi_t + \pi)}$ from
$0$, this would imply that $\operatorname{diam}g_t(\mathcal{C}) >
1/2$ and this is a
contradiction.
\end{pf}

%pr3.2 #&#
\begin{Proposition}\label{nov4.1}
Let $T <\infty$ be given. There exist constants $0< c_1$, $c_2$, \mbox{$c_3 <
\infty$} with $c_1$ depending only on $T$ and $c_2,c_3$ universal such that
the following holds. Let $\gamma$ be a curve in $\mathbb{D}$
generating the
Loewner pair $(f, W)$ and let $\eta_{\mathrm{tip}}(\delta)$ be
the tip
structure modulus for $(\gamma(t), t \in[0,T])$. Then if $t \in[0,T]$
and $\Delta_t(d) < c_1$, we have
%
%e32 #&#
\begin{equation}
\bigl|\gamma(t) - f_t\bigl((1-d)W_t\bigr)\bigr| \le
c_2\eta_{\mathrm{tip}} \bigl(c_3 \Delta
_{t}(d) \bigr).
\end{equation}
\end{Proposition}

\begin{pf}
We use the notation from the proof of Lemma~\ref{1027}.
Set
\[
\delta_0=\rho_1 \Delta/2,
\]
where $\rho_1$ is as in Lemma~\ref{1027}.
Then by Lemma~\ref{1027} (if $\Delta< c_1$, where $c_1$ is the
constant of that lemma) there is a crosscut $\mathcal{C}\subset
\mathcal{B}(z_d,
\delta_0)$ separating $z_d$ and $\gamma(t)$ from $0$ in $D_t$ while
$\operatorname{diam}\mathcal{C}\le2\delta_0$. By the definition of
tip structure modulus,
$\operatorname{dist}(\gamma(t), \mathcal{C}) \le\eta
_{\mathrm{tip}} (2\delta_0 )$
and consequently, $|z_d-\gamma(t)| \le\eta_{\mathrm{tip}} (2\delta_0
) + \delta_0$.\vadjust{\goodbreak}
\end{pf}

One can also estimate the distance to the tip directly in terms of $d$,
the distance to the boundary in $\mathbb{D}$.

%
%pr3.3 #&#
\begin{Proposition}\label{wolff}
There is a constant $c<\infty$ such that the following holds. Let $T <
\infty$ be given. Let $\gamma$ be a curve in $\mathbb{D}$ generating the
Loewner pair $(f, W)$ and let $\eta_{\mathrm{tip}}(\delta)$ be
the tip
structure modulus for $(\gamma(t), t \in[0,T])$. Then for every $t
\in
[0,T]$ and $d <1/2$,
%
%e33 #&#
\begin{equation}
\bigl|\gamma(t) - f_t\bigl((1-d)W_t\bigr)\bigr| \le c
\eta_{\mathrm{tip}} \bigl(\bigl(2\pi A/(\log 1/d)\bigr)^{1/2} \bigr),
\end{equation}
where $A$ may be chosen as $\min\{\pi(\operatorname{diam}\gamma
_{0,T})^2,\pi\}$.
\end{Proposition}

\begin{pf}
The needed estimate is a consequence of a classical result due to
\mbox{J.~}Wolff. We will give a short proof using extremal length. Consider
\mbox{$\mathcal{A}=\mathcal{A}(r, R) \cap\mathbb{D}$} centered around $W_t$,
the preimage of $\gamma(t)$ in $\partial{\mathbb{D}}$. Let $E$~and~$F$
be the two boundary components of $\mathcal{A}$ which are contained in
$\partial{\mathbb{D}}$. By comparing with a half-annulus and mapping to
a rectangle, using also the comparison principle for extremal length,
we see that the extremal distance between $E$ and $F$ in~$\mathcal{A}$
is at most $\pi/\log(R/r)$. Hence, by conformal invariance and the
definition of extremal length,
\[
\frac{\pi}{\log(R/r)} \ge\frac{L^2}{A},
\]
where $L$ is the euclidean length of the curve-family connecting $f(E)$
with $f(F)$ in $f(\mathcal{A})$ and $A$ is the Euclidean area of
$f(\mathcal{A})$. The number $A$ is clearly bounded above by the
minimum of $\pi(\operatorname{diam}\gamma_{0,T})^2$ and $\pi$.
Consequently, by
taking $r=d$ and $R=\sqrt{d}$ we see that there exists a crosscut
$\mathcal{C}
'$ of $D_t$ separating $\gamma(t)$ and $z_d=f_t((1-d)W_t)$ from $0$ and
the diameter of $\mathcal{C}'$ is at most $l(d):= (2\pi A/(\log
1/d) )^{1/2}$.
Hence, $\operatorname{dist}(\gamma(t), \mathcal{C}') \le\eta
_{\mathrm{tip}}(l(d))$ and an argument
using the Gehring--Hayman theorem (see, e.g., Theorem~4.20 of \cite
{Pom92}, and also below) now shows that $\operatorname{dist}(z_d,
\gamma(t)) \le c
(\eta_{\mathrm{tip}}(l(d)) + l(d)) \le c' \eta_{\mathrm{tip}}(l(d))$.
\end{pf}

We end the section with a lemma that combines some of the previous work
in this section and that of Section~\ref{prel}. It is tailored for the
situation where a discrete model Loewner curve approaches an SLE curve
in the scaling limit. We will use it in the proof of Theorem~\ref{lerw-thm} in Section~\ref{lerw-sect}.

%le3.4 #&#
\begin{Lemma}\label{dec3.2}
For $j=1,2$, let $(f_j, W_j)$ be $\mathbb{D}$-Loewner pairs generated
by the curves $\gamma_j$. Fix $T < \infty$ and $\rho>1$. Assume that
there exist $\beta< 1$, $r \in(0,1)$, $p \in(0, \frac{1}{\rho})$ and
$\varepsilon> 0$ such that the following holds with
\[
d_* = \varepsilon^p.\vspace*{-20pt}
\]
\begin{longlist}[(iii)]
\item[(i)] The driving terms satisfy
\[
\sup_{t \in[0,T]}\bigl|W_1(t)-W_2(t)\bigr| \le
\varepsilon;
\]

\item[(ii)]
There exists a constant $c < \infty$ such that the tip structure
modulus for $(\gamma_1(t), t \in[0, T])$ in $\mathbb{D}$ satisfies
\[
\eta_{\mathrm{tip}}(d_*) \le c d_*^r;
\]

\item[(iii)]
There exists a constant $c' < \infty$ such that the derivative estimate
\[
\sup_{t \in[0,T]} \,d\bigl|f_2'
\bigl(t,(1-d)W_2(t)\bigr)\bigr| \le c' d^{1-\beta}\qquad \forall d \le d_*,
\]
holds.
\end{longlist}
Then there is a constant $c''=c''(T, \beta, r, p, c, c') < \infty$
such that
\[
\sup_{t \in[0,T]}\bigl|\gamma_1(t)-\gamma_2(t)\bigr|
\le c'' \max\bigl\{ \varepsilon ^{p(1-\beta)r},
\varepsilon^{(1-\rho p)r}\bigr\}.
\]
The analogous statement holds for $\mathbb{H}$-Loewner pairs.
\end{Lemma}

\begin{pf}
The proof is immediate from the assumptions using Proposition~\ref
{1031} combined with Proposition~\ref{nov4.1}.
\end{pf}

%s3.3 #&#
\subsection{H\"older regularity}\label{sec3.3}
We shall now see that the John-type condition $\eta_{\mathrm{tip}}(\delta) \le
A \delta,   \delta< \delta_0$, forces a curve driven by a H\"older
continuous function to be H\"older continuous in the capacity
parameterization, with exponent depending only on $A$ and the exponent
for the driving term. Note that we must have $A \ge1$. We will derive
a bound on the growth of the derivative as in (\ref{der-eq}) from the
bound on $\eta_{\mathrm{tip}}$. H\"older regularity then
follows from
Proposition~\ref{der}. The proof uses the length--area principle. The
situation is different from the classical one; see, for example, \cite
{W} or \cite{Pom92}, in that our assumptions do not prevent large
bottlenecks to form.

%
%th3.5 #&#
\begin{Theorem}\label{geom-holder}
Suppose that the radial Loewner pair $(f,e^{i\xi})$ is generated by a
curve $\gamma$. Assume that $\xi$ is H\"older continuous and that there
exist $A<\infty$ and $\delta_0>0$ such that the tip structure modulus
for $(\gamma(t),   t \in[0, T])$ in $\mathbb{D}$ satisfies $\eta
_{\mathrm{tip}
}(\delta) \le A\delta$, $\delta< \delta_0$. Then $\gamma$ is H\"older
continuous on $[0,T]$ with H\"older exponent depending only on $A$ and
the H\"older exponent for $\xi$.
\end{Theorem}

%re7 #&#
\begin{Remark*}
A bound on the tip structure modulus alone cannot imply H\"older
regularity of the path in the capacity parameterization; it is
necessary to have some regularity of the driving term. Indeed, consider
the chordal setting and take $\gamma$ to be the graph of $e^{-1/x},
x \in[0,1]$. For this curve, the tip structure modulus clearly decays
linearly, uniformly in $t$. On the other hand, parameterize by
half-plane capacity and note that there is a universal constant $c$
such that
\[
2 t =\operatorname{hcap}\gamma[0,t] \le c \mbox{ height }\gamma[0,t]
\cdot\operatorname{diam}\gamma[0,t].
\]
(This follows, e.g., from a harmonic measure estimate.)
Hence,
\[
t \le c e^{-1/\operatorname{Re}\gamma(t)}\operatorname{Re}\gamma(t),
\]
which shows that $\gamma$ is not H\"older continuous at $t=0$. (By
precomposing with slit map $\sqrt{z^2-4T}$, a similar example can be
constructed with the ``singularity'' occurring at an arbitrary $T>0$.)
Moreover, if $W$ is the driving term for $\gamma$, then
\[
\operatorname{diam}\gamma[0,t] \asymp\sqrt{t} + \sup_{s \in[0,t]}\bigl|W(s)\bigr|,
\]
so $W$ is also not H\"older continuous. (In fact, a similar argument
shows that if the driving term is H\"older-$\alpha$, $\alpha\le1/2$,
at $t=0$, then so is the curve.)

It is possible to take this example as a starting point to formulate a
geometric condition that implies H\"older continuity for the driving
term. We shall not, however, pursue this further here.
\end{Remark*}

Before giving the proof of Theorem~\ref{geom-holder}, we need a simple lemma.

%le3.6 #&#
\begin{Lemma}\label{stolz}
Let $f\dvtx \mathbb{D} \to D$ be a conformal map with $f(0)=0$.
Define the Stolz cone
\[
S_r=\bigl\{1-\rho e^{i\theta}\dvtx  0\le\rho\le r, -\pi/4 \le\theta
\le \pi /4 \bigr\}.
\]
There is a universal constant $c<\infty$ such that
\[
\operatorname{diam}f(S_r) \le c \operatorname{diam}f(
\sigma_r),
\]
where $\sigma_r=[1-r,1)$ is the line segment connecting $1-r$ and $1$.
\end{Lemma}

\begin{pf}
Let $u=1-\rho e^{i\theta}$ be an arbitrary point in $S_r$. By Koebe's
distortion theorem, there is a universal constant $c$ such that
\[
\bigl|f(u) - f(1-\rho)\bigr| \le c \rho\bigl|f'(1-\rho)\bigr|.
\]
Hence, by Koebe's estimate there is a universal constant $c'$ such that
\begin{eqnarray*}
\bigl|f(u) - f(1-\rho)\bigr| & \le& c' \operatorname{dist}\bigl(f(1-\rho),
\partial D\bigr)
\\
& \le& c' \operatorname{diam}f(\sigma_r)
\end{eqnarray*}
and this completes the proof.
\end{pf}

\begin{pf*}{Proof of Theorem~\ref{geom-holder}}
Let $t \in[0,T]$ and write $W_t=e^{i\xi_t}$. Without loss of
generality, we may assume that $t>0$ and that $W_t=1$. We suppress the
dependence on $t$ and write $f$ for $f_t$ and $D$ for $D_t$, etc. throughout the proof. Set $z_r=f(1-r)$ and $\Delta_r = \operatorname
{dist}(z_r,
\partial D)$. By Proposition~\ref{wolff}, there is an $r_0$ depending
only on $A$~and~$\delta_0$ such that $\Delta_r \le\delta_0$ for all $r
\le r_0$. By taking $r_0$ smaller if necessary, depending only on $T$,
we can guarantee that the assumptions of Lemma~\ref{1027} are satisfied
so that there will exist a universal $\rho_0<\infty$ and a crosscut
$\mathcal{C}
$ contained in $\partial\mathcal{B}(z_r, \rho_0 \Delta_r)$ that separates
$z_r$ and $\gamma(t)$ from $0$ in $D$. Let $\sigma_r = [1-r,1]$. We
claim that $f(\sigma_r)$, which connects $z_r$ with $\gamma(t)$ in $D$,
satisfies
%
%e34 #&#
\begin{equation}
\label{1028} \operatorname{diam}f(\sigma_r) \le c
\rho_0 A \Delta_r,
\end{equation}
where $c$ is a universal constant.
To prove this, note that since $\mathcal{C}$ separates $\gamma(t)$~and~$z_r$
from $0$, the hyperbolic geodesic $f(\sigma_1) \supset f(\sigma_r)$
which connects $\gamma(t)$ and $0$ must intersect $\mathcal{C}$.
[Since $\gamma
$ is a Loewner curve, $\gamma(t)$ is always on the boundary of the
simply connected domain $D_t \ni0$.] Let $\Gamma''$ be the curve
obtained by tracing $f(\sigma_1)$ from $0$ to $\gamma(t)$ until
$\mathcal{C}$
is first hit. Let $\Gamma' = f(\sigma_1) \setminus\Gamma''$. Then
$\Gamma'$ is a hyperbolic geodesic connecting a point on $\mathcal
{C}$ with
$\gamma(t)$ in $D_t$ and $f(\sigma_r) \subset\Gamma'$. By the bound on
the structure modulus, there is a curve $\Gamma$ connecting $\gamma(t)$
with $\mathcal{C}$ in~$D_t$ and
\[
\operatorname{diam}\Gamma\le2 A \operatorname{diam}\mathcal{C}\le 4
\rho_0 A \Delta_r.
\]
The Gehring--Hayman theorem (see, e.g., Chapter~4 of \cite{Pom92}) now
implies that there is a universal constant $c$ such that
\[
\operatorname{diam}f(\sigma_r) \le\operatorname{diam}
\Gamma' \le c (\operatorname{diam}\Gamma+ \operatorname{diam}
\mathcal{C})
\]
and this gives (\ref{1028}).

Using Lemma~\ref{stolz}, the remainder of the proof now proceeds by a
standard length--area type argument (see, e.g., Chapter~5 of \cite
{Pom92}). Define
\[
\varphi(r) = \int_0^r \bigl|f'(1-r)\bigr|^2
r \,dr.
\]
Then by Koebe's distortion theorem, there is a universal constant $c_0$
such that
%
%e35 #&#
\begin{equation}
\label{lb} r^2\bigl|f'(1-r)\bigr|^2 \le
c_0 \int_{r/2}^r r
\bigl|f'(1-r)\bigr|^2  \,dr \le c_0 {\varphi}(r).
\end{equation}
This theorem also implies that there is a constant $c_1$ depending only
on $c_0$ such that
\[
\varphi(r) \le c_1 \int_0^r
\int_{-\pi/4}^{\pi/4} \bigl|f'
\bigl(1-re^{i\theta}\bigr)\bigr|^2 r \,dr \,d\theta= c_1
\operatorname{area}f(S_r),
\]
where $S_r$ is the Stolz cone defined in the statement of Lemma~\ref{stolz}. Now, by (\ref{1028}) and Lemma~\ref{stolz} we have that
\[
\operatorname{area}f(S_r) \le\frac{\pi^2}{4} \bigl(\operatorname
{diam}f(S_r)\bigr)^2 \le c_2
\Delta_r^2.
\]
Hence,
\[
\varphi(r) \le c_1 \operatorname{area}f(S_r) \le
c_3 r^2\bigl|f'(1-r)\bigr|^2.
\]
Consequently, since $\varphi'(r)=r |f'(1-r)|^2$, we have for $r_0 > r$ and
a constant $c_4$ depending only on $A$
\[
\log \biggl( \frac{\varphi(r_0)}{\varphi(r)} \biggr) = \int_r^{r_0}
\frac{\varphi
'(r)}{\varphi(r)}  \,dr \ge c_4^{-1} \log \biggl(
\frac{r_0}{r} \biggr).
\]
Taking exponentials, using (\ref{lb}), gives for $0 < r \le r_0$
\[
r^2\bigl|f'(1-r)\bigr|^2 \le c_5
r^{1/c_4},
\]
where $c_5$ depends only on $r_0$. Hence, if $\beta=1-1/(2c_4) < 1$ we
see that
\[
r\bigl|f'(1-r)\bigr| \le c_6 r^{1-\beta}.
\]
By Proposition~\ref{der}, since the estimates were uniform in $t$, this
implies H\"older regularity with an exponent depending only on $A$ and
the exponent for $W$.
\end{pf*}

%s4 #&#
\section{Loop-erased random walk and SLE$_2$}\label{lerw-sect}
This section proves a convergence rate result for loop-erased random
walk using the setup detailed in the previous sections.

%s4.1 #&#
\subsection{Definitions}\label{sec4.1}The radial \textit{Schramm--Loewner evolution},
radial SLE$_\kappa$, is defined by taking $W(t) = e^{i\sqrt{\kappa}
B(t)}$ as driving term for the radial Loewner equation, where $B$ is
standard Brownian motion. It is a fact that this Loewner chain is
almost surely generated by a curve---the SLE$_\kappa$ path. This is a
random fractal curve which is simple when $0 \le\kappa\le4$, has
double points when $4 < \kappa$ and is space filling when $\kappa\ge
8$. See \cite{RS} for proofs of these results. In Appendix~\ref
{sle-sect}, we discuss a derivative estimate for radial SLE$_\kappa$
that we will state and use in this section when $\kappa=2$. For
technical reasons, we need a stopping time $\sigma$ for the radial SLE
path $\tilde{\gamma}$ further discussed in Appendix~\ref{sle-sect}. Fix
a small constant $\varepsilon>0$. We then define
%
%e36 #&#
\begin{equation}
\label{sigma1} \sigma=\sigma(\varepsilon, T)=\inf\bigl\{t \ge0\dvtx
\bigl|g_t(-1)-W(t)\bigr| \le \varepsilon\bigr\} \wedge T,
\end{equation}
where $g_t=f_t^{-1}$ is the forward Loewner SLE$_2$ flow and $W(t)$ is
the driving term for $f_t$.

%
%pr4.1 #&#
\begin{Proposition}\label{may61}
Let $\varepsilon>0$ and $T < \infty$ be fixed and let $(f_s),   0
\le s \le
\sigma$, be the stopped radial SLE$_2$ Loewner chain with $\sigma
=\sigma
(\varepsilon,T)$ defined by (\ref{sigma1}). For every $\beta\in
(2(\sqrt {10}-1)/9,1)$ and $q < q(\beta)$, there exists a constant $c=c(\beta,q, \varepsilon,T)< \infty$ such for all $d_* \le1$
\[
\mathbb{P} \Bigl\{\forall d \le d_*, \sup_{s \in[0, \sigma
]} \,d\bigl|f'_s
\bigl((1-d)W(s)\bigr)\bigr| \le d^{1-\beta} \Bigr\} \ge1-cd_*^{q},
\]
where
\[
q(\beta)=-1+2\beta+\frac{\beta^2}{4(1+\beta)}.
\]
\end{Proposition}

\begin{pf}
See Appendix~\ref{sle-sect}.
\end{pf}

Let $D \ni0$ be a simply connected domain and assume that the inner
radius of~$D$ with respect to $0$ equals $1$. We will assume, for
simplicity, that $D$ is a Jordan domain with $C^{1+\alpha}$ boundary,
where $\alpha>0$. We shall consider a particular discretization of $D$.
A \textit{grid-domain} with respect to $ n^{-1} \mathbb{Z}^2$ is a simply
connected domain whose boundary is a subset of the edge set of the
graph $ n^{-1} \mathbb{Z}^2$. We define $D_n=D_n(D)$, the $ n^{-1}
\mathbb{Z}^2$
grid-domain approximation of $D$, as the component of~$0$ of~$\mathbb
{C}$ minus those closed $n^{-1} \mathbb{Z}^2$ lattice faces that intersect
$\partial D$. Then clearly $D_n$ is a grid-domain contained in $D$. Let
$\psi_n \dvtx  D_n \to\mathbb{D}$ be the normalized conformal map.

Suppose $S=S(j)$,  $j=0,1,\ldots,m$, is a finite nearest-neighbor walk
on (the vertices of $n^{-1}\mathbb{Z}^2$ contained in) $D_n$. We define
the loop-erasure $\mathcal{L}\{S\} \subset S$ in the following way. If
$S$ is already self-avoiding,
set $\mathcal{L}\{S\}=S$. Otherwise, let $s_0 = \max\{j \dvtx  S(j)=S(0)\}$,
and for $i > 0$,
let $s_i = \max\{j \dvtx  S(j) = S(s_{i-1}+1) \}$. If we let $n = \min\{i \dvtx
s_i=m\}$, then
$\mathcal{L}\{S\} = \{S(s_0), S(s_1), \ldots, S(s_n)\}$. Notice that
$\mathcal{L}\{S\}(0)=S(0)$ and $\mathcal{L}\{S\}(s_n) = S(m)$, that is,
the loop-erased walk has the same end points as the original walk $S$.
\textit{Loop-erased random walk} (LERW) from $0$ to $\partial D_n$ in
$D_n$ is the random self-avoiding walk $\gamma_n$ obtained by taking
$S$ to be a simple random walk on $n^{-1} \mathbb{Z}^2$ started from
$0$ and
stopped when reaching $\partial D_n$, and then setting $\gamma
_n=\mathcal{L}\{S\}$. For a nearest-neighbor walk $S$, let $S^R$ be the
time-reversed walk. It is known that LERW has the following symmetry
with\vspace*{1pt} respect to time-reversal: the distribution of $(\mathcal{L}\{S\}
)^R$ is equal to that of $\mathcal{L}\{S^R\}$. Sometimes it is more
convenient to consider $\mathcal{L}\{S^R\}$, and when we do we will
call it the \textit{time-reversed LERW} (or time-reversal of LERW) and
usually assume that the path is traced from the boundary toward $0$; we
always add edges in the obvious way to discrete walks to make them curves.

%s4.2 #&#
\subsection{Convergence rate for the LERW path}\label{sec4.2}Lawler, Schramm and
Werner proved in \cite{LSW} that, as $n \to\infty$, the image of the
time-reversed LERW path in~$\mathbb{D}$, $\psi_n ( \mathcal
{L}\{ S^R \}
)$, traced from $\partial D$ toward $0$, converges weakly with
respect to a natural metric on curves modulo increasing
reparameterization toward the radial SLE$_2$ path started uniformly on
$\partial D$. (See Theorem~3.9 of \cite{LSW} for a precise statement.)
The goal of this section is to prove Theorem~\ref{lerw-thm}, which can
be viewed as a quantitative version of Theorem~3.9 of \cite{LSW}.

Let $D$ be a simply connected $\mathcal{C}^{1+\alpha}$ domain with grid-domain approximation $D_n=D_n(D)$. Let $\gamma_n$ be the time-reversal
of LERW on $n^{-1} \mathbb{Z}^2$ from $0$ to $\partial D_n$ and let
$\tilde{\gamma}_n=\psi_n(\gamma_n)$ be its image in $\mathbb{D}$
traced from
the boundary and parameterized by capacity. (Since $\gamma_n$ is a
simple curve that intersects $\partial D_n$ at only one point it
follows that $\tilde{\gamma}_n$ is a $\mathbb{D}$-Loewner curve for
each $n$.)
Let $W_n(t)$ be the Loewner driving term for $\tilde{\gamma}_n$. Fix $s
\in(0,1/24)$, and define
\[
\varepsilon_n=n^{-s}.
\]

%th4.2 #&#
\begin{Theorem}[(\cite{BJK})]\label{lerw-thm0}
For every $T> 0$, there exists
$n_0=n_0(T,s)< \infty$ such that the following holds. For each $n \ge
n_0$, there is a coupling of $\gamma_{n}$ with Brownian motion $B(t)$,
$t\ge0$, where $e^{iB(0)}$ is uniformly distributed on the unit circle,
with the property that
%
%e37 #&#
\begin{equation}
\label{the.equation2} \mathbb{P} \Bigl\{ \sup_{t \in[0, T]}
\bigl|W_n(t)-W(t)\bigr| > \varepsilon_n \Bigr\} <
\varepsilon_n,
\end{equation}
where $W(t)=e^{iB(2t)}$.
\end{Theorem}

%re8 #&#
\begin{Remark*}
The coupling(s) of $W_n=e^{i\theta_n}$ and $W=e^{iB}$ in Theorem~\ref
{lerw-thm0} are via Shorokhod embedding of $\theta_n$ into $B$.
\end{Remark*}

We can now state a precise version of the main result of the paper.

%
%th4.3 #&#
\begin{Theorem}\label{lerw-thm}
There exists $n_1=n_1(\varepsilon, T, s) < \infty$ such that if $n
\ge n_1$,
then in the coupling of Theorem~\ref{lerw-thm0}, if $\tilde{\gamma}$
denotes the radial SLE$_2$ path in $\mathbb{D}$ driven by $W$,
%
%e38 #&#
\begin{equation}
\label{the.equation3} \mathbb{P} \Bigl\{ \sup_{t \in[0, \sigma]} \bigl|\tilde{\gamma
}_n(t)-\tilde\gamma (t)\bigr| > \varepsilon_n^{m}
\Bigr\} < \varepsilon_n^{m},
\end{equation}
where both curves are parameterized by capacity,
\[
m=1/41
\]
and $\sigma=\sigma(\varepsilon, T)$ is the stopping time defined by
(\ref{sigma1}).
\end{Theorem}

%re9 #&#
\begin{Remark*}
The proof of Theorem~\ref{lerw-thm} (with minor modifications) would
also work under the weaker assumption that $D$ is a quasidisk. (The
class of quasidisks includes, e.g., the von Koch snowflake.) In this
case, the rate would depend on the constant in the Ahlfors three-point
condition satisfied by $\partial D$; see Appendix~\ref{grid-sect}. We
may also note that the conclusion (and proof) of Theorem~\ref{lerw-thm}
holds true in any coupling like the one of Theorem~\ref{lerw-thm0},
with the\vspace*{1pt} proviso that $\varepsilon_n$ decays slower than~$n^{-1/2}$.
\end{Remark*}

%
%re10 #&#
\begin{Remark*}
By Lemma~\ref{dec15.1} below, the preimages of the curves
(parameterized by capacity) in $D_n$ satisfy a similar estimate as in
(\ref{the.equation3}), namely,
\[
\mathbb{P} \Bigl\{ \sup_{t\in[0,\sigma]} \bigl|\gamma_n(t)-\psi
_n^{-1} \bigl(\tilde\gamma(t) \bigr)\bigr| >
\varepsilon_n^{m} \Bigr\} < \varepsilon_n^{m},
\qquad m=1/41.
\]
\end{Remark*}

In order to apply the work from previous sections, we need to verify
that the assumptions of these results hold with large probability. In
Section~\ref{lerw-sm-sect}, we will first estimate the probability of
the existence of a certain power-law bound for the tip structure
modulus for the LERW path in $D_n$. We show in Appendix~\ref{grid-sect}
that if $\partial D$ is sufficiently smooth ($\mathcal{C}^{1+\alpha}$),
then the image of the LERW path in $\mathbb{D}$ enjoys the same tip
structure modulus up to constants. This uses a convergence rate result
for grid-domain approximations of quasidisks that we derive from a
result of Warshawski's. In Appendix~\ref{sle-sect}, we prove the needed
estimate on the derivative of the SLE$_2$ conformal maps. These results
are combined to prove Theorem~\ref{lerw-thm} in Section~\ref{ps}.

%s4.3 #&#
\subsection{Tip structure modulus for LERW in a grid domain}\label{lerw-sm-sect}
An important tool to get quantitative estimates for LERW is the \textit{Beurling estimate} for simple random walk; see, for example, \cite
{LL}. There are many ways to formulate this result and we state only
one version here.

%
%le4.4 #&#
\begin{Lemma}
There exists a constant $c < \infty$ such that the following holds. Let
$A \subset\mathbb{Z}^2$ be an infinite connected set. Let $S$ be
simple random walk on $\mathbb{Z}^2$ started from $z$ and stopped at
the time $\tau_A$ at which $S$ hits $A$. Then for $r >1$
\[
\mathbb{P} \bigl\{ \bigl|S(\tau_A) - z\bigr| \ge r \operatorname{dist}(z, A)
\bigr\} \le cr^{-1/2}.
\]
\end{Lemma}

We can now formulate the main estimate of this section.

%
%pr4.5 #&#
\begin{Proposition}\label{lerw-sm}
Let $D_n$ be a grid domain with respect to $n^{-1}\mathbb{Z}^2$ and
assume that $1 \le\operatorname{inrad}(D_n) \le2$ and that
$\operatorname{diam}D_n \le R< \infty
$, where $R$ is given. Let $\gamma_n$ be the time-reversal of
loop-erased random\vspace*{1pt} walk from $0$ to $\partial D_n$. Let $\eta
_{\mathrm{tip}
}^{(n)}(\delta)$ be the tip structure modulus for $\gamma_n$ (traced
from $\partial D_n$) stopped when first reaching distance $\varepsilon
>0$ from
$0$. Let $r \in(0,1/11)$. There exists a universal constant $c_0 > 0$
and $c=c(R,r, \varepsilon)< \infty$ such that if $n$ is sufficiently
large and
$\delta> c_0/n$, then
%
%e39 #&#
\begin{equation}
\mathbb{P} \bigl\{\eta_{\mathrm{tip}}^{(n)}(\delta)\le
\delta^r \bigr\} \ge 1- c\delta^{1/5-11r/5}|\log\delta|.
\end{equation}
\end{Proposition}

%
%re11 #&#
\begin{Remark*}
When we apply Proposition~\ref{lerw-sm}, we will choose $\delta
=\delta
(n) \in\omega(n^{-1})$ (in the sense of Landau notation) so that
$\delta> c_0/n$ is automatically satisfied for $n$ sufficiently large.
\end{Remark*}

%
%re12 #&#
\begin{Remark*}
The Beurling estimate implies that there is a constant $c< \infty$
such that
\[
\mathbb{P}\{\operatorname{diam}\gamma_n > R\} \le c
R^{-1/2}
\]
for large $R$. This means that one can formulate and prove
Proposition~\ref{lerw-sm} with an estimate independent of the diameter
of $D_n$.
\end{Remark*}

%
%s4.4 #&#
\subsection{Proof of Proposition~\texorpdfstring{\protect\ref{lerw-sm}}{4.5}}\label{sec4.4}
The result was formulated for the time-reversal of LERW but in the
proof we shall consider the LERW generated by erasing the loops of
simple random walk from $0$ to $\partial D_n$ (without the
time-reversal). By time-reversal symmetry, this is sufficient.

The strategy of the proof is based on that of the proof of Lemma~3.4 in
\cite{SchrammIJM}, but see also the related Lemma~3.12 of \cite{LSW}.
See Figure \ref{fig3} for a sketch of different crossing configurations that may occur.
Let $w$ be a fixed point in $D_n$. Let $\mathcal{A}=\mathcal{A}(w;
\delta, \eta)=\{z\dvtx  \delta< |z-w| < \eta\}$ be the $(\delta, \eta
)$-annulus about $w$ and assume (for now) that $\delta> 10/n$ and we
think of $\eta$ as much larger than $\delta$ but still small compared
to $\operatorname{inrad}D$; eventually, we want to choose $\eta=
\delta^r$ for some
$r \in(0,1)$. Let $\gamma$ be a curve in $D_n$. We say that $\gamma$
has a $k$-crossing of the annulus $\mathcal{A}$ if the number of
components of $\gamma\cap\mathcal{A}$ that connect the two boundary
components of~$\mathcal{A}$~is at least $k$.
Recall that $\eta(\delta)$ is a bound for the tip structure modulus for
$\gamma$ in $D_n$ if and only if $\gamma$ has no nested $(\delta,\eta
(\delta))$-bottleneck in $D_n$. Now consider $\gamma_n$, the LERW path
in $D_n$ traced from $\partial D_n$ toward $0$ and the event that
there is a nested $(\delta, 2\eta)$-bottleneck in $\gamma_n$ stopped
when reaching $\partial\mathcal{B}(0,\varepsilon)$. We claim that
this event is
contained in the union of the following two events:
\begin{enumerate}[$\mathcal{E}_5={}$]
\item[$\mathcal{E}_5=$]{$\!\!$\{There is a $w \in D_n$ with $|w| >
\varepsilon$ such
that $\gamma_n$ has a $5$-crossing of a \mbox{$(\delta, \eta)$-}annulus about
$w$\}.}
\item[$\mathcal{E}_B=$]{$\!\!$\{The random walk generating $\gamma_n$ travels
more than distance $\eta$ before hitting $\partial D_n$, after the
first time it has come within distance $\delta$ from $\partial D_n$\}.}
\end{enumerate}
%
%f3 #&#
\begin{figure}[t]

\includegraphics{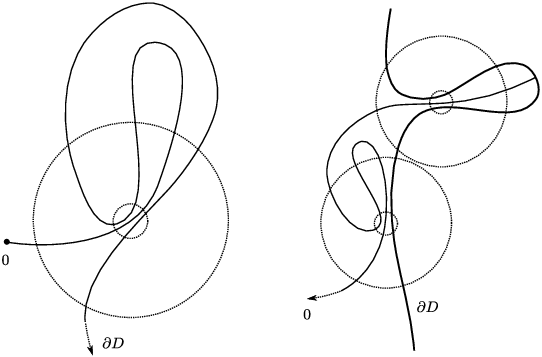}

\caption{A $6$-crossing and crossings close to $\partial D$.}\label{fig3}
\end{figure}
Indeed, suppose that a nested $(\delta, 2\eta)$-bottleneck occurs in
$\gamma_n$ stopped when reaching $\partial\mathcal{B}(0,\varepsilon
)$. Then if we
choose some parameterization of $\gamma_n$ traced from $\partial D_n$
to $0$, by definition there exist $t_0$ and a crosscut $\mathcal{C}$
of $D'=D_n
\setminus\gamma[0,t_0]$ such that $\operatorname{diam}\mathcal
{C}\le\delta$ and $\operatorname{diam}
\gamma_\mathcal{C}\ge2\eta$. Consider first the case when
$\overline{C} \cap
\partial D_n \neq\varnothing$. Then since $\gamma_n$ connects
$\partial
D_n$ with $0$ and $\mathcal{C}$ separates a piece of $\gamma_n$ from
$0$ we
must have that $\gamma_n$ intersects $\mathcal{C}$. Consequently, the random
walk that generates $\gamma_n$ intersects $\mathcal{C}$, and if
$\mathcal{C}$ is to
separate a piece of $\gamma_n$ of diameter at least $2\eta$ the event
$\mathcal{E}_B$ must occur.

Now suppose that $\overline{\mathcal{C}} \cap\partial D_n =
\varnothing$. We
will show that this implies that $\mathcal{E}_5$ must occur. Notice
that $D'\setminus\mathcal{C}$ consists of two simply connected
components, one
of which
has no part of its boundary in common with $\partial D_n$. Call this
component $\mathcal{O}$. There are two cases: first, assume that $0
\notin\mathcal{O}$. Then $\gamma_\mathcal{C}\subset\mathcal{O}$
and so
$\operatorname{diam}\mathcal{O} \ge2\eta$. By considering $\partial
\mathcal{O}
\setminus(\mathcal{C}\cup\gamma_\mathcal{C})$ (giving two
crossings) and $\gamma_\mathcal{C}
$ traced from $\mathcal{C}$ to $\gamma_n(t_0)$ and then continued
along $\gamma
_n$ to $0$ (giving three crossings) we see that $\gamma_n$ indeed
contains a $5$-crossing of $(\delta, \eta)$-annulus. On the other hand,
if $0 \in\mathcal{O}$ we have that $\mathcal{B}(0, \varepsilon)
\subset\mathcal{O}$
so $\operatorname{diam}\mathcal{O} \ge2\eta$ if $\eta< \varepsilon
/2$. In this case,
$\gamma_\mathcal{C}\subset D' \setminus\mathcal{O}$ and again considering
$\partial\mathcal{O} \setminus(\mathcal{C}\cup\gamma_\mathcal
{C})$ and $\gamma_\mathcal{C}$
traced from $\mathcal{C}$ to $\gamma_n(t_0)$ and then continued along
$\gamma
_n$ to $0$, we see that $\gamma_n$ contains a $5$-crossing of a
$(\delta,\eta)$-annulus.

We will estimate the probabilities of the two events $\mathcal{E}_5$
and $\mathcal{E}_B$, starting with the last. In this case, the Beurling
estimate immediately implies that there is a constant $c<\infty$ such that
%
%e40 #&#
\begin{equation}
\label{EC} \mathbb{P} (\mathcal{E}_B ) \le c \biggl(
\frac{\delta
}{\eta} \biggr)^{1/2}.
\end{equation}
We proceed to bound $\mathbb{P} (\mathcal{E}_5 )$. Fix a point
$w \in D_n$ with $|w| > \varepsilon$. Set
\[
d_0=\operatorname{dist}(w, \partial D_n) > 0
\]
and define
\[
\mathcal{B}_1=\mathcal{B}(w, \eta/4),\qquad\mathcal{B}_2=
\mathcal {B}(w, \eta/2).
\]
For a curve $\gamma\subset D_n$, we let $\mathcal{Q}^3(\gamma; w,
\delta, \eta)$ denote the event that $\gamma$ has a $3$-crossing of a
$(\delta, \eta)$-annulus whose smaller boundary component is contained
in $\mathcal{B}_1$. Similarly, let $\mathcal{Q}^5(\gamma; w, \delta, \eta)$
denote the event that $\gamma$ has a $5$-crossing of a $(\delta, \eta
)$-annulus whose smaller boundary component is contained in $\mathcal{B}_1$.
Clearly, the latter event is contained in the former. We will first
estimate the probability of
\[
\mathcal{Q}^5:=\mathcal{Q}^5(\gamma_n; w,
\delta, \eta).
\]
Let $S(t)=S_n(t),   t=0,1,\ldots, \tau$, be the simple random walk
generating $\gamma_n$; it is started from $0$ and stopped at
\[
\tau=\min\bigl\{t \ge0 \dvtx  S(t) \in\partial D_n\bigr\},
\]
when $\partial D_n$ is hit. Define
\[
s_1=\min\bigl\{t \ge0 \dvtx  S(t) \in\mathcal{B}_1\bigr
\}, \qquad t_1=\min\bigl\{t > s_1 \dvtx  S(t) \notin
\mathcal{B}_2\bigr\}
\]
and recursively for $j=2,3, \ldots,$
\[
s_j=\min\bigl\{t > t_{j-1} \dvtx  S(t) \in\mathcal{B}_1
\bigr\}, \qquad t_j=\min\bigl\{t > s_j \dvtx  S(t) \notin
\mathcal{B}_2\bigr\}.
\]
Note that we have $s_1=0$ if $|w| \le\eta/4$ and $s_1 > 0$ otherwise.
We will write
\[
\mathcal{Q}_j^5:=\mathcal{Q}^5\bigl(
\mathcal{L}\bigl\{S[0, t_j]\bigr\}; w, \delta, \eta \bigr), \qquad
\mathcal{Q}^3_j:=\mathcal{Q}^3\bigl(
\mathcal{L}\bigl\{S[0, t_j]\bigr\}; w, \delta, \eta\bigr).
\]
Clearly, $\mathcal{Q}^5_j \subset\mathcal{Q}_j^3$, but it does not
necessarily hold that $\mathcal{Q}^5_{j+1} \subset\mathcal{Q}^5_j$ or
$\mathcal{Q}^3_{j+1} \subset\mathcal{Q}^3_{j}$ because part of the
curve forming a crossing may be erased.
Note that for $m \ge1$
\[
\mathbb{P} \bigl( \mathcal{Q}^5 \bigr) \le\mathbb{P}\{\tau>
t_{m+1}\} + \mathbb{P} \Biggl(\bigcup_{j=1}^m
\mathcal{Q}_j^5 \Biggr).
\]
We estimate $\mathbb{P}\{\tau> t_{m+1}\}$ in Lemma~\ref{feb1612} below.

We have
\[
\mathbb{P} \Biggl(\bigcup_{j=1}^m
\mathcal{Q}_j^5 \Biggr) \le\sum
_{j=1}^m\mathbb{P} \bigl( \mathcal{Q}_j^5, \lnot\mathcal{Q}_{j-1}^5 \bigr).
\]
To get the last estimate, we split the event on the left-hand side
according to the first time a $5$-crossing has occurred; here and in
the sequel, for an event $A$ the symbol~``$\lnot A$'' means the
complement of $A$.
To\vspace*{1.5pt} bound $\mathbb{P}(\mathcal{Q}_j^5,   \lnot\mathcal{Q}_{j-1}^5
)$, let
us first discuss the analogous quantity for a $3$-crossing.
In the proof of Lemma~3.4 of \cite{SchrammIJM} [on p. 241, after
equation~(3.4)], it was essentially shown that there is a (nonrandom)
constant $c< \infty$ such that
%
%e41 #&#
\begin{equation}
\label{ql0} \mathbb{P} \bigl( \mathcal{Q}_j^3 |
\lnot\mathcal{Q}_{j-1}^3, S[0,t_{j-1}] \bigr) \le
c (j-1) \biggl(\frac{\delta}{\eta} \biggr)^{1/2}.
\end{equation}
The exponent in the right-hand side of (\ref{ql0}) was not specified in
\cite{SchrammIJM} so let us sketch the proof and explain how one gets
the exponent $1/2$. Let $\{C_k\}_k$ be the components of $\mathcal{L}\{
S[0, s_j]\} \cap\mathcal{B}_2$ intersecting $\mathcal{B}_1$ but not
containing
$S(s_{j})$. By construction, there are at most $j-1$ such components.
Conditionally, on $S[0,t_{j-1}]$, if $\mathcal{L}\{S[0,t_j] \}$ is to
contain a $3$-crossing which was not there in $\mathcal{L}\{
S[0,t_{j-1}] \}$, then $S[s_j, t_j]$ has to come within distance
$\delta
$ of $C_k \cap\mathcal{B}_1$ for some $k$ and then exit $\mathcal
{B}_2$ without
hitting that same $C_k$. (It may hit other components.) For each
component~$C_k$, we can use the strong Markov property and the Beurling
estimate to see that this conditional probability of exiting $\mathcal{B}_2$
without hitting $C_k$ is bounded above by $c (\delta/\eta)^{1/2}$.
Summing over the $j-1$ components gives (\ref{ql0}).

From (\ref{ql0}),
%
%e42 #&#
\begin{equation}
\label{ql1} \mathbb{P} \bigl( \mathcal{Q}_j^3 |
\lnot\mathcal{Q}_{j-1}^3 \bigr) \le c (j-1) \biggl(
\frac{\delta}{\eta} \biggr)^{1/2}.
\end{equation}
And this implies that
%
%e43 #&#
\begin{equation}
\label{eq:revjune3.1} \mathbb{P} \bigl( \mathcal{Q}_j^3 \bigr)
\le\sum_{k=1}^j \mathbb {P}
\bigl(Q_k^3, \lnot Q_{k-1}^3
\bigr) \le c j^2 \biggl( \frac{\delta}{\eta} \biggr)^{1/2}.
\end{equation}
We now turn to $\mathbb{P} ( \mathcal{Q}_j^5,   \lnot\mathcal
{Q}_{j-1}^5  )$. Since $(\mathcal{Q}_j^5 \cap\lnot\mathcal
{Q}_{j-1}^5) \subset\mathcal{Q}_{j-1}^3$, (\ref{eq:revjune3.1}) implies
\begin{eqnarray*}
\mathbb{P} \bigl( \mathcal{Q}_j^5, \lnot
\mathcal{Q}_{j-1}^5 \bigr) & =& \mathbb{P} \bigl(
\mathcal{Q}_j^5, \lnot\mathcal{Q}_{j-1}^5
| \mathcal {Q}_{j-1}^3 \bigr) \mathbb{P} \bigl(
\mathcal{Q}_{j-1}^3 \bigr)
\\
& \le& c \mathbb{P} \bigl(\mathcal{Q}_j^5, \lnot\mathcal
{Q}_{j-1}^5 | \mathcal{Q}_{j-1}^3
\bigr) j^2 \biggl( \frac{\delta}{\eta} \biggr)^{1/2}.
\end{eqnarray*}
We continue to write
\[
\mathbb{P} \bigl(\mathcal{Q}_j^5, \lnot
\mathcal{Q}_{j-1}^5 | \mathcal {Q}_{j-1}^3
\bigr) \le\mathbb{P} \bigl(\mathcal{Q}_j^5 |\lnot
\mathcal {Q}_{j-1}^5, \mathcal{Q}_{j-1}^3
\bigr).
\]
We can estimate the last expression by observing that
\[
\mathbb{P} \bigl(\mathcal{Q}_j^5 |\lnot
\mathcal{Q}_{j-1}^5, \mathcal {Q}_{j-1}^3,
S[0, t_{j-1}] \bigr) \le c (j-1) \biggl(\frac{\delta
}{\eta}
\biggr)^{1/2}.
\]
Indeed, this estimate is proved in exactly the same way as (\ref{ql1})
using the Beurling estimate.

Combining our bounds, we get
%
%e44 #&#
\begin{equation}
\mathbb{P} \Biggl(\bigcup_{j=1}^m
\mathcal{Q}_j^5 \Biggr) \le c m^4
\frac
{\delta}{\eta}.
\end{equation}
We now take $\nu>0$ and let $m=\lfloor\delta^{-\nu} \rfloor$. We then
use Lemma~\ref{feb1612} (here we write the estimate for $d_0 > \eta/4$;
in the case $d_0 \le\eta/4$ we use the second bound of Lemma~\ref
{feb1612}) to get
%
%e45 #&#
\begin{eqnarray}\label{bulk}
\mathbb{P}\bigl(\mathcal{Q}^5\bigr) &\le& \biggl(1-
\frac{c_3}{|\log(16
 d_0/\eta
)|} \biggr)^{\lfloor\delta^{-\nu} \rfloor} + c \frac{\delta
^{1-4\nu
}}{\eta}
\nonumber\\[-8pt]\\[-8pt]
& \le& c \delta^{\nu} \bigl| \log(16  d_0/\eta)\bigr| + c
\frac{\delta
^{1-4\nu
}}{\eta}.\nonumber
\end{eqnarray}
This bound is for a fixed $w$. To conclude, note that there is a
universal $c < \infty$ such that we can (deterministically) cover $D_n$
using at most $c R^2\eta^{-2}$ overlapping disks $\mathcal{B}(w_k,
\eta/4)$
in such a way for every $w$ such that $\gamma_n$ has a $5$-crossing of
$\mathcal{A}(w; \delta, \eta)$, the smaller boundary component of
$\mathcal{A}(w; \delta, \eta)$ is contained in $\mathcal{B}(w_k,
\eta/4)$ for
some $k$. Consequently, for $c=c(R)<\infty$,
%
%e46 #&#
\begin{equation}
\label{E5} \mathbb{P} ( \mathcal{E}_5 ) \le c \eta^{-2}
\delta^{\nu
} \bigl| \log (16  d_0/\eta)\bigr| + c \eta^{-3}
\delta^{1-4\nu}.
\end{equation}

For any $r \in(0,1/11)$, if $\eta=\delta^{r}$, we can take $\nu
=(1-r)/5$ in (\ref{E5}) which makes both terms in the bound of the same
(``polynomial'') order so that the right-hand side of (\ref{E5}) decays
like $\delta^{1/5-11r/5}$ with a logarithmic correction. Since this
term is always larger than the one coming from $\mathcal{E}_B$, this
completes the proof of Proposition~\ref{lerw-sm}, assuming Lemma~\ref{feb1612}. %\qed

%
%le4.6 #&#
\begin{Lemma}\label{feb1612}
There exist constants $0 < c_1, c_2 < 1$ such that
\[
\mathbb{P}\{\tau> t_{m+1}\} \le\cases{
\displaystyle \biggl(1- \frac{c_1}{|\log(16  d_0\eta^{-1})|}
\biggr)^{m}, &\quad if $d_0 > \eta/4$;
\vspace*{5pt}\cr
\displaystyle (1-c_2)^m, &\quad if $d_0 \le\eta/4$.}
\]
\end{Lemma}

\begin{pf}
We first assume that $d_0 > \eta/4$. Using, for example,
Proposition~6.4.1 of \cite{LL}, we see that the probability that a
simple random walk started just outside of $\mathcal{B}_2$ exits
$\mathcal{B}(z_0,
8 d_0)$ before hitting $\mathcal{B}_1$ is bounded below by
\[
\frac{|\log2| - O((\eta n )^{-1})}{|\log(16  d_0\eta^{-1})|} \ge \frac
{|\log2|}{2|\log(16  d_0\eta^{-1})|}
\]
if $\eta n > c_1$, where $c_1 < \infty$ is a universal constant. (This
uses also that $d_0 > \eta/4$.)
This estimate is a discrete version of the expression for the harmonic
measure of one of the boundary components in an annulus. Moreover,
there is a universal constant $c>0$ such that the probability that
simple random walk from (a vertex adjacent to) $\partial\mathcal{B}(z_0,
8 d_0)$ separates $\mathcal{B}(z_0, d_0)$ from $\infty$ before hitting
$\mathcal{B}
(z_0, d_0)$ is bounded below by $c$. (Recall that our assumptions imply
that $d_0 > c'/n$, where we can assume that $c'$ is large.)
Consequently, by the strong Markov property the probability that simple
random walk started from $\partial\mathcal{B}_2$ exits $D_n$ before hitting
$\mathcal{B}_1$ is bounded below by $c_1/|\log(16 d_0\eta^{-1})|$.
By iterating this argument using the strong Markov property,
%
%e47 #&#
\begin{equation}
\label{nov12.1} \mathbb{P}\{\tau> t_{m+1}\} \le \biggl(1-
\frac{c_1}{|\log(16
 d_0\eta
^{-1})|} \biggr)^{m}.
\end{equation}
When $d_0 \le\eta/4$ the Beurling estimate and the Markov property
directly show that the right-hand side of (\ref{nov12.1}) can be
replaced by $ (1- c_2 )^{m}$, where $c_2 > 0$ is a universal
constant.
\end{pf}

If the boundary of the domain $D$ that is being approximated is
sufficiently regular, then the structure modulus on a sufficiently
large mesoscopic scale for the image curve in $\mathbb{D}$ is
essentially the same as the one in $D_n$. The next lemma, proved in
Appendix~\ref{grid-sect}, makes this precise.

%
%le4.7 #&#
\begin{Lemma}\label{dec15.1}
Suppose $D \ni0$ is a simply connected domain Jordan domain with
$C^{1+\alpha}$ boundary, where $\alpha> 0$. Let $D_n$ be the $n^{-1}
\mathbb{Z}^2$ grid-domain approximation of $D$ and let $\gamma_n$ be
a Loewner
curve in $D_n$ connecting $\partial D_n$ with $0$. There is a constant
$c$ depending only on $\alpha$ and the diameter of $D$ such that the
following holds. Set $0< r <1/2$ and $d_n=n^{-r}$ and let $\eta
_{\mathrm{tip}
}^{(n)}(\delta; D_n)$ be the\vspace*{1pt} tip structure modulus for $\gamma_n$ in
$D_n$. Then for all $n$ sufficiently large (independently of $\gamma
_n$) the tip structure modulus $\eta_{\mathrm{tip}}^{(n)}(\delta; \mathbb{D})$
for $\psi_n(\gamma_n)$ in $\mathbb{D}$ satisfies
\[
\eta_{\mathrm{tip}}^{(n)}\bigl( c^{-1}  d_n;
\mathbb{D}\bigr) \le c \eta _{\mathrm{tip}}^{(n)}(d_n;
D_n).
\]
\end{Lemma}

%s4.5 #&#
\subsection{Proof of Theorem~\texorpdfstring{\protect\ref{lerw-thm}}{4.3}}\label{ps}
We write $\gamma$ for the radial SLE$_2$ path in $\mathbb{D}$
corresponding to the Brownian motion in (\ref{the.equation2}). We thus
have a coupling of the radial SLE$_2$ path and the image of the LERW
path $\tilde{\gamma}_n$ and we will estimate the distance between these
curves in this coupling. Take $s \in(0,1/24)$ and $n > n_0$ where
$n_0$ is as in Theorem~\ref{lerw-thm0}; fix $\rho>1$ and for $p \in
(0,1/\rho)$, let
\[
\varepsilon_n = n^{-s}, \qquad d_n=(
\varepsilon_n)^p.
\]
For each $n \ge n_0$, we shall define three events each of which occurs
with large probability in our coupling. On the intersection of these
events, we can apply our estimates from Sections~\ref{prel} and \ref{geom}.
\begin{longlist}[(a)]
\item[(a)]{
Let $\mathcal{A}_n=\mathcal{A}_n(s)$ be the event that the estimate
\[
\sup_{t \in[0,T]} \bigl|W_n(t)-W(t)\bigr| \le
\varepsilon_n
\]
holds. By Theorem~\ref{lerw-thm0}, we know that there exists $n_0 <
\infty$ such that if $n \ge n_0$ then
\[
\mathbb{P}(\mathcal{A}_n) \ge1- \varepsilon_n.
\]
}
\item[(b)]{For $\beta\in(2(\sqrt{10}-1)/9,1)$, let $\mathcal
{B}_n=\mathcal{B}_n(s,r, \beta, \varepsilon, T, c_B)$ be the event
the radial
SLE$_2$ Loewner chain $(f_t)$ driven by $W(t)$ satisfies the estimate
\[
\sup_{t\in[0,\sigma]}  \,d \bigl|f'\bigl(t,(1-d)W(t)\bigr)\bigr| \le
c_B  d^{1-\beta}\qquad\forall d \le d_n.
\]
(Recall that $\varepsilon, T$ were used in the definition of the stopping-time
$\sigma\le T$.) Then by Proposition~\ref{may61} there exist $c_B' <
\infty$, independent of $n$, and $n_1 < \infty$ such that if $n \ge
n_1$ then
\[
\mathbb{P} (\mathcal{B}_n ) \ge1-c'_B
 d_n^{q},
\]
where
\[
q<q_2(\beta)=-1+2\beta+ \frac{\beta^2}{4(1+\beta)}.
\]
}
\item[(c)]{For $r \in(0, 1/11)$, let $\mathcal{C}_n=\mathcal{C}_n(s,
r, p, c_C, \alpha, \operatorname{diam}D)$ be the event that the tip structure
modulus for $\tilde{\gamma}_n(t),   t \in[0,T]$, in $\mathbb{D}$,
$\eta_{\mathrm{tip}
}^{(n)}$, satisfies
\[
\eta_{\mathrm{tip}}^{(n)}(d_n)\le c_C
 d_n^{r}.
\]
We know from Proposition~\ref{lerw-sm} and Lemma~\ref{dec15.1} that
there exist $c_C, c_C' < \infty$, independent of $n$, and $n_2 <
\infty
$ such that if $n\ge n_2$ then
\[
\mathbb{P}(\mathcal{C}_n) \ge1-c_C'
d_n^{1/5-11r/5}|\log d_n|.
\]
}
\end{longlist}
Consequently, there exist $c_B, c_C < \infty$ and $c < \infty$, all
independent of $n$ (but depending on $s,r,p,\varepsilon, T, \beta,
\alpha,
\operatorname{diam}D$), such that for all $n$ sufficiently large,
%
%e48 #&#
\begin{equation}
\label{dec2.1} \mathbb{P} ( \mathcal{A}_n \cap\mathcal{B}_n
\cap\mathcal {C}_n ) \ge1 -c \bigl(\varepsilon_n+d_n^{q}
+ d_n^{1/5-11r/5}|\log d_n|\bigr)
\end{equation}
and on the event $\mathcal{A}_n \cap\mathcal{B}_n \cap\mathcal{C}_n$
we can apply Lemma~\ref{dec3.2} with constants $c=c_C$, $c'=c_B$
independent of $n$ to see that there exists $c''$ independent of $n$
(but depending on the above parameters) such that for all $n$
sufficiently large,
%
%e49 #&#
\begin{equation}
\label{dec2.3} \sup_{t \in[0,\sigma]}\bigl|\tilde{\gamma}_n(t)-
\tilde{\gamma}(t)\bigr| \le c'' \bigl(d_n^{r(1-\beta)}
+ \varepsilon_n^{(1-\rho p)r}\bigr).
\end{equation}
We now wish to optimize over the parameters in the exponents.
Since $d_n=\varepsilon_n^p$, we see that $d_n^{r(1-\beta)}$ dominates
in (\ref{dec2.3}) when $p \in(0, 1/(1+\rho-\beta)]$ and
$\varepsilon_n^{r(1-\rho p)}$
whenever $p \in[1/(1+\rho-\beta),1]$. Suppose $ p \in(0, 1/(1+\rho
-\beta)]$.

Set
\[
\mu(\beta,r)=\min \biggl\{r(1- \beta), -1+2\beta+ \frac{\beta
^2}{4(1+\beta)},
\frac{1}{5}-\frac{11r}{5} \biggr\}.
\]
The optimal rate is given by optimizing $\mu$ over $\beta,r$ and then
choosing $p$ very close to $1/(1+\rho-\beta)$. (No improvement is
obtained by considering $p \in[1/(1+\rho-\beta),1]$.) Let $\beta_*
\in
(2(\sqrt{10}-1)/9,1)$ be a solution to
\[
45\beta^3-128\beta^2-84\beta+ 68=0.
\]
(One can check that $\beta_* = 0.497\ldots.$) Then if $r_*=1/(16-\beta
_*) \in(0,1/11)$
\[
\mu(r_*,\beta_*)=\max \biggl\{\mu(\beta,r)\dvtx  \frac{2(\sqrt {10}-1)}{9} < \beta<1, 0 < r<
\frac{1}{11} \biggr\}=0.037\ldots.
\]
Consequently, for every
\[
m < m_*=\frac{\mu(r_*,\beta_*)}{2-\beta_*},
\]
we obtain bounds in (\ref{dec2.1}) and (\ref{dec2.3}) of order
$\varepsilon
_n^m$ for all $n$ sufficiently large. Since $1/41< m_* =0.024\ldots,$
this completes the proof. %\qed

\begin{appendix}
\section{Derivative estimate for radial SLE}\label{sle-sect}
This section proves a derivative estimate for both chordal and radial
SLE. The radial case was needed in Section~\ref{lerw-sect} in the case
$\kappa=2$. The chordal case is a direct consequence of an estimate
from \cite{JVL}, but the radial case requires a little bit of work. In
this case, our goal will be to estimate explicitly in terms of $d_*$~and~$\beta$ the probability of the event that when $(f(t,z))$ is the
radial SLE$_\kappa$ Loewner chain, the estimate $d|f'(t,(1-d) W(t))|
\le c   d^{1-\beta}$ for all $d \le d_*$ holds uniformly in $t \in[0,
T]$. This will follow from a moment estimate for the chordal reverse
flow in~\cite{JVL} after changing ``coordinates'' from radial to
chordal SLE. See also Section~7 of~\cite{BJK} where a similar but
nonequivalent situation is dealt with. We will use ideas from~\cite{schramm-wilson}.

%s5.1 #&#
\subsection{Change of coordinates}\label{sigma}
Let $(f_s,W_s)$ be a radial Loewner pair generated by the curve $\gamma
(s)$ with $W_s$ continuous. Recall that $f_s\dvtx  \mathbb{D} \to\mathbb{D}
\setminus K_s=D_s$ and that $K_s$ is the hull generated by $\gamma
[0,s]$. Let $g_s = f_s^{-1}$ and set $z_s=g_s(-1)\overline{W_s}$. We
will need to keep track of the ``disconnection time'' $\sigma'$ when
$K_s$ first disconnects $-1$ from $0$ in $\mathbb{D}$, in other words, the
first time that $z_s$ hits $1$. Fix $\varepsilon>0$ small and $T <
\infty$, and define
\begin{equation}\label{eq41}
 \sigma=\sigma(\varepsilon,T)=\inf \bigl\{s \ge0\dvtx  |1- z_s|
\le \varepsilon \bigr\} \wedge T.
\end{equation}
Clearly, $\sigma< \sigma'$.

%
%le5.1 #&#
\begin{Lemma}\label{may12.1}
There exists a constant $c=c(\varepsilon,T)>0$ such that
\[
\inf_{s \in[0,T]}\bigl|g'_{s \wedge\sigma}(-1)\bigr| \ge c.
\]
\end{Lemma}

\begin{pf}
The Loewner equation implies that with $z_s$ as above,
\[
\bigl|g'_s(-1)\bigr| = \exp \biggl\{ \int_0^s
\operatorname{Re}\frac
{2}{(1-z_s)^2}-1  \,ds \biggr\}.
\]
This shows that $|g'_s(-1)|$ is strictly decreasing in $s$ and that
$|g'_{T \wedge\sigma}(-1)| \ge c=c(\varepsilon,T) >0$.
\end{pf}

%
%re13 #&#
\begin{Remark*}
Note that if $g_s$ is the radial SLE$_\kappa$ forward
flow, and if
\[
\theta_s:=-i \log z_s = -i \log g_s(-1) -
\sqrt{\kappa}B_s, \qquad \theta _0 = \pi,
\]
then by It\^o's formula,
\[
d \theta_s = \cot(\theta_s/2)  \,ds - \sqrt{\kappa}  \,d
B_s.
\]
If $\kappa< 4$, then it follows from \cite{lawler-book}, Lemma 1.27,
that almost surely $\theta_s$ does not hit $\{0, 2\pi\}$ in finite
time. Hence, for each $T <\infty$, if $\kappa< 4$, then almost surely,
\[
\lim_{\varepsilon\to0} \sigma(\varepsilon,T) =T.
\]
\end{Remark*}

Consider now the Mobius transformation
\[
\varphi\dvtx  \mathbb{H} \to\mathbb{D}, \qquad\varphi(z)=\frac{i-z}{i+z}.
\]
Then $\varphi^{-1} \circ\gamma$ is a curve in $\mathbb{H}$ (for
sufficiently small $s$) and for $s \ge0$ we define
\[
t(s):= \operatorname{hcap}\bigl(\varphi^{-1}\bigl(\gamma[0,s]\bigr)
\bigr)/2.
\]
For each $s \in[0, \sigma]$, let $F_{t(s)}\dvtx  \mathbb{H} \to
H_{t(s)}:=\varphi^{-1}(D_s)$ be the conformal mapping satisfying the
hydrodynamical normalization $F_{t(s)}(z) = z - 2t(s)/z + o(1/|z|)$ at
infinity. It is known (see, e.g., \cite{schramm-wilson}) that $t(s)$ is
a strictly increasing, continuous function of $s$ up to the
disconnection time and we will write $s(t)$ for the inverse of~$t(s)$.
One can write (see \cite{schramm-wilson} and \cite{BJK})
%
%e50 #&#
\begin{equation}
\label{dec16.1} f_s = \varphi\circ F_{t(s)} \circ
\Delta_s.
\end{equation}
Here,
%
%e51 #&#
\begin{equation}
\label{dec16.2} \Delta_s(z)\dvtx  \mathbb{D} \to\mathbb{H}, \qquad
\Delta_s(z) = \frac{z
\overline{\mu_{t(s)}} - \lambda_s \mu_{t(s)}}{z-\lambda_s},
\end{equation}
where the reader may verify that if
\[
G_{t(s)}(z)= F_{t(s)}^{-1}(z), \qquad
g_s(z) = f_s^{-1}(z),
\]
then
\[
\mu_{t(s)} = G_{t(s)}(i), \qquad\lambda_s =
g_s(-1).
\]
In fact, by expanding $G$ at infinity via (\ref{dec16.1}),
%
%e52 #&#
\begin{equation}
\label{dec16.3} \operatorname{Im}\mu_{t(s)} = -\frac{g'_s(-1)}{g_s(-1)}=\bigl|g_s'(-1)\bigr|.
\end{equation}
This uses that
\[
\operatorname{Re} \biggl(1 - \frac{g_s''(-1)}{g_s'(-1)} \biggr) = -\frac{g'_s(-1)}{g_s(-1)},
\]
which holds because the left-hand side equals $\partial_\theta[ \arg
\partial_\theta g_s(e^{i\theta}) ]$ at $\theta=\pi$, and $g_s$ maps the
circle to the circle locally at $-1$ so that the change of the tangent
is equal to the change of the argument which is what is represented by
the right-hand side. By Lemma~\ref{may12.1} and (\ref{dec16.3}) there
exists $c_1=c_1(\varepsilon,T)>0$ such that
%
%e53 #&#
\begin{equation}
\label{mu-lb} \operatorname{Im}\mu_{t(s)} \ge c_1, \qquad
s \in[0, \sigma].
\end{equation}
Set
\[
\tau:=t(\sigma)
\]
and consider the family $(F_t),   t \in[0, \tau]$, with the half-plane
capacity parameterization. It satisfies the chordal Loewner PDE in $t$
and we let $U_t=\Delta_{s(t)}(W_{s(t)})$ be the corresponding chordal
driving term. The estimate (\ref{mu-lb}) implies that there is
$T'=T'(\varepsilon,T) < \infty$ such that $\tau\le T'$. Indeed, in Theorem~3
of \cite{schramm-wilson} it is shown that $s'(t) = 4 (\operatorname
{Im}\mu
_{s(t)})^2/|\mu_{s(t)}-U_t|^4$ which is bounded away from $0$ on
$[0,\tau]$. Using (\ref{mu-lb}) and that $|W_s-\lambda_s| \ge
\varepsilon$ for
$s \in[0, \sigma]$, we see that there exist constants $0 < c < \infty$
and $d_0>0$ depending only on $\varepsilon$ and $T$ such that for all
$d \le
d_0$, uniformly in $s \in[0,\sigma]$,
\[
\bigl\llvert \operatorname{Re} \bigl(\Delta_s\bigl((1-d)W_s
\bigr) \bigr) -U_{t(s)}\bigr\rrvert \le c d, \qquad c^{-1}  \,d \le
\operatorname{Im} \bigl(\Delta_s\bigl((1-d)W_s\bigr)
\bigr) \le c d.
\]
In other words, the hyperbolic distance between $\Delta_s((1-d)W_s)$
and $U_{t(s)} + id$ is bounded by a constant depending only on
$\varepsilon$
and $T$. Therefore, we can use Koebe's distortion theorem to see that
there exist $c, c'<\infty$ depending only on $\varepsilon,T$ such
that for all
$s \in[0,\sigma]$
\[
\bigl|f_s'\bigl((1-d)W_s\bigr)\bigr| \le c
\bigl|F'_{t(s)}\bigl(\Delta_s\bigl((1-d)W_s
\bigr)\bigr)\bigr| \le c' \bigl|F'_{t(s)}(U_{t(s)}
+ id)\bigr|.
\]
We have proved the following result.

%
%pr5.2 #&#
\begin{Proposition}\label{dec16.4}
Let $T < \infty$ and $\varepsilon>0$ be given. Suppose that
$(f_s,W_s)$ is a
radial Loewner pair generated by the curve $\gamma(s)$. Define $\sigma
=\sigma(\varepsilon,T)$ by (\ref{eq41}). Let $(F_t, U_t)$ be the chordal
Loewner pair generated by the curve $s \mapsto\varphi^{-1}(\gamma
(s))$, $s \in[0,\sigma]$ reparameterized by half-plane capacity and let $\tau
=t(\sigma)$. There exists $c=c(\varepsilon,T) < \infty$ and
$d_0=d_0(\varepsilon,T)>0$
such that for all $d \le d_0$,
\[
\sup_{s\in[0, \sigma]} \bigl|f'_s
\bigl((1-d)W_s\bigr)\bigr| \le c \sup_{t \in[0, \tau
]}\bigl|F'_{t}(U_{t}
+ id)\bigr|.
\]
\end{Proposition}

Now assume that $(f_s)$ is the radial SLE$_\kappa$ Loewner chain. Then
$\sigma$ is a stopping time for $(f_s)$ and $\tau$ is a stopping time
for $(F_t)$. The law of the chordal driving term $U_t$ stopped at $\tau
$ is absolutely continuous with respect to the law of standard linear
Brownian motion with speed $\kappa$, as shown in \cite{schramm-wilson}.
Moreover, by (\ref{mu-lb}) the Girsanov density is uniformly bounded
above by a constant depending only on $\kappa,\varepsilon$ and $T$.
Indeed, it
is a product of powers of $|G'_t(i)|$, $\operatorname{Im}\mu_t$, and
$|\mu_t -
U_t|$, all which are bounded away from $0$ and $\infty$ when $t \le
\tau$.
Since $(F_t)$ is absolutely continuous with respect to a chordal
SLE$_\kappa$ Loewner chain and since the Girsanov density is uniformly
bounded (for fixed $\kappa, \varepsilon, T$), using Proposition~\ref{dec16.4}
we can estimate the behavior of $\sup_{s\in[0,\sigma
]}|f_s'((1-d)W_s)|$ using standard chordal SLE.

%s5.2 #&#
\subsection{Derivative estimate for chordal SLE}\label{sec5.2}
We now derive the needed estimate on the growth of the derivative in
chordal coordinates. The estimate is essentially a direct consequence
of work in \cite{JVL} and we will describe the modifications here. Let
$(F_t),   t \ge0$, be the standard chordal SLE Loewner chain mapping
$\mathbb{H}$ onto the unbounded connected component of $\mathbb
{H}\setminus\gamma[0,t]$. We write $\widehat{F}_t(z) = F_t(z+U_t)$, where
$U$ is the chordal driving term for $(F_t)$.
Recall that the chordal reverse SLE$_\kappa$ flow is the family of
conformal mappings solving
\[
\dot{h}_t=-\frac{2}{h_t-\sqrt{\kappa}B_t}, \qquad h_0(z)=z,
\]
where $B$ is standard Brownian motion. For fixed $t_0 > 0$,
$|h_{t_0}'(z)|$ is equal to $|\widehat{F}'_{t_0}(z)|$ in distribution.
Hence, (first) moment estimates for $|\widehat F'_{t_0}|$ are reduced to
corresponding estimates for $|h'_{t_0}|$ and these are often more
easily obtained.\vspace*{-2pt} Note that scaling implies that for fixed $y>0$,
$|h_{t}'(iy)| \stackrel{d}{=} |h_{ty^{-2}}'(i)|$. Define
\[
\zeta(\lambda)=\lambda+ \frac{\sqrt{(4 + \kappa)^2-8\lambda
\kappa
}-(4+\kappa)}{4}.
\]
We will assume that
\[
\lambda< \lambda_c = 1+\frac{2}{\kappa}+ \frac{3\kappa}{32}.
\]
In this range, we quote the following estimate from \cite{JVL}. See
also \cite{JVL2} and the references therein.

%le5.3 #&#
\begin{Lemma}\label{ub}
Let $h_t$ be the chordal reverse SLE$_\kappa$ flow, $\kappa> 0$. There
exists a constant $c <\infty$ such that for $\lambda< \lambda_c$.
%
%e54 #&#
\begin{equation}
\mathbb{E}\bigl[\bigl|h_{t}'(i)\bigr|^{\lambda}\bigr] \le c
t^{-\zeta(\lambda)/2}, \qquad t \ge1.
\end{equation}
\end{Lemma}

This result now implies the needed estimate which is a version of
Proposition~4.2 of \cite{JVL} with a decay rate; we will sketch the
proof and refer the reader to \cite{JVL} for more details.
Let $\kappa> 0$ and define the function
\[
\rho(\beta)=\beta+\frac{2(1+\beta)}{\kappa}+ \frac{\beta^2
\kappa
}{8(1+\beta)}
\]
and
\[
q(\beta)=\min\bigl\{\lambda_c \beta, \rho(\beta)-2\bigr\}, \qquad
\beta_+ < \beta<1,
\]
where
\[
\beta_+=\max \biggl\{ 0, \frac{4(\kappa\sqrt{8+\kappa}-(4-\kappa
))}{(4+\kappa)^2} \biggr\}.
\]
Note that $q(\beta)>0$ for $\beta$ in the above range.

%
%pr5.4 #&#
\begin{Proposition}\label{may4}
Let $T < \infty$ be fixed and let $(F_t)$ be the chordal SLE$_\kappa$
Loewner chain, $\kappa\in(0,8)$. Let $\beta\in(\beta_+,1)$ and $q <
q(\beta)$. There exists a constant $0< c< \infty$ depending only on
$T,\kappa,q$ such that for every $y_* < 1$
\[
\mathbb{P} \Bigl\{ \forall y \le y_*, \sup_{t \in[0,
T]}y\bigl|\widehat{F}'_t(iy)\bigr| \le c y^{1-\beta} \Bigr\}
\ge1-cy_*^{q}.
\]
\end{Proposition}

\begin{pf}(Sketch.)
By the distortion theorem, scaling and the fact that Brownian motion is
almost surely weakly H\"older-$(1/2)$, it is enough (see \cite{JVL}) to
show that for $\beta_+ < \beta< 1$ and $q < q(\beta)$
\[
\sum_{n=N_*}^{\infty} \sum
_{j=1}^{2^{2n}}\mathbb{P}\bigl(\bigl|\widehat{F}'_{j2^{-2n}}\bigl(i2^{-n}\bigr)\bigr| >
2^{\beta n}\bigr) \le c 2^{-N_*q},
\]
where $N_*=\lfloor\log y_*^{-1} \rfloor$.
We have for $0<\lambda< \lambda_c$ using scaling, Chebyshev's
inequality and Lemma~\ref{ub},
\begin{eqnarray*}
&& \sum_{n=N_*}^{\infty} \sum
_{j=1}^{2^{2n}}\mathbb{P}\bigl(\bigl|\widehat
{F}'_{j2^{-2n}}\bigl(i2^{-n}\bigr)\bigr| >
2^{\beta n}\bigr)
\\
&&\qquad  \le\sum_{n=N_*}^{\infty}
\sum_{j=1}^{2^{2n}} 2^{-n \lambda\beta}
\mathbb{E}\bigl[ \bigl|\widehat {F}'_{j2^{-2n}}\bigl(i2^{-n}
\bigr)\bigr|^{\lambda}\bigr]
\le c\sum_{n=N_*}^{\infty} \sum
_{j=1}^{2^{2n}} 2^{-n \lambda\beta} \mathbb{E}\bigl[
\bigl|h'_{j}(i)\bigr|^{\lambda}\bigr]
\\
&&\qquad \le c\sum_{n=N_*}^{\infty} \sum
_{j=1}^{2^{2n}} 2^{-n \lambda\beta} j^{-\zeta/2}
\le c\sum_{n=N_*}^{\infty} \sum
_{j=1}^{2^{2n}} 2^{-n \lambda\beta} \bigl(1+2^{n(2-\zeta)}
\bigr)
\\
&&\qquad \le c \bigl(2^{-N_* \lambda\beta} + 2^{-N_*(\lambda\beta+ \zeta-2)}\bigr).
\end{eqnarray*}
Recall that $\lambda\in(0,\lambda_c)$. Note that $\zeta-2 < 0$ if and
only if $\kappa>1$, so for these $\kappa$ the smaller exponent is
$\lambda\beta+ \zeta-2$. In this range, we find $q(\beta)$ by
maximizing over $0<\lambda< \lambda_c$ for $\beta$ fixed so that
$q(\beta)=\max_{\lambda} \lambda\beta+\zeta(\lambda)-2$. The lower
bound $\beta_+$ is the smallest $\beta> 0$ such that $\beta> \beta_+$
implies $q(\beta)>0$. When $\kappa\le1$, $\lambda\beta$ is the
smaller exponent and we must restrict attention to $\beta>0$. We pick
the largest $\lambda=\lambda_c$.
\end{pf}

From this and the work in the previous subsection, we immediately
obtain the following proposition. Recall that the stopping time $\sigma
$ was defined in (\ref{eq41}).

%
%pr5.5 #&#
\begin{Proposition}\label{may6}
Let $\kappa\in(0,8)$. Let $\varepsilon>0$ be fixed and let $(f_s),
0 \le
s \le\sigma$, be the radial SLE$_\kappa$ Loewner chain stopped at
$\sigma$ as defined by (\ref{eq41}). For every $\beta\in(\beta_+,1)$
and $q < q(\beta)$, there exists a constant $c=c(\beta, \kappa, q,
\varepsilon,
T)< \infty$ such that for $d_* < 1$,
\[
\mathbb{P} \Bigl\{\forall d \le d_*, \sup_{s \in[0, \sigma
]} \,d\bigl|f'_s
\bigl((1-d)W_s\bigr)\bigr| \le c d^{1-\beta} \Bigr\}
\ge1-cd_*^{q}.
\]
\end{Proposition}

We note that when $\kappa=2$
\[
q(\beta)=-1+2\beta+\frac{\beta^2}{4(1+\beta)}, \qquad\beta _+=\frac
{2(\sqrt{10}-1)}{9}.
\]

%s6 #&#
\section{Mapping to $\mathbb{D}$}\label{sec6}\label{grid-sect}
When mapping conformally a curve into a reference domain, bounds on the
tip structure modulus for the curve are not automatically preserved. In
this section, we will consider a general case without reference to a
specific discrete model. It seems that this general setting requires
information about boundary regularity of the approximated domain (as
opposed to information about the behavior of the discrete curve). In
particular, we will need uniform control of the distortion of annuli on
the scales of the structure modulus.

%s6.1 #&#
\subsection{Grid domains}\label{sec6.1}Recall the definition of a grid domain that
was given in Section~\ref{lerw-sect}. Let $D \ni0 $ be simply
connected, and assume that the inner radius with respect to $0$ equals
$1$. Let $D_n=D_n(D)$ be the $n^{-1} \mathbb{Z}^2$ grid-domain approximation
of $D$. Notice that every point on $\partial D_n$ is within distance
$\sqrt{2}/n$ of a point on~$\partial D$, so that the inner Hausdorff
distance between $\partial D_n$ and $\partial D$ is at most $\sqrt{2}/n$.
Let $\psi\dvtx  D \to\mathbb{D}$ be the conformal map normalized by $\psi
(0)=0$ and $\psi'(0) > 0$. Similarly, for $n=1,2,\ldots,$ let $\psi_n \dvtx
D_n \to\mathbb{D}$ be conformal maps with the same normalization. The
sequence of domains $D_n$ converge to $D$ in the Carath\'eodory sense,
and so the $\psi_n$ converge to $\psi$ uniformly on compacts. Our goal
will be to find a convergence rate for
\[
\sup_{z \in D_n}\bigl|\psi_n(z)-\psi(z)\bigr|.
\]
For this to be achievable, we need some information about the
regularity of the boundary of $D$. We will here consider the class of
quasidisks, although it will be clear that similar methods can be used
to handle other classes of domains (e.g., John domains) where Euclidean
geometric estimates on the behavior of the conformal mapping on the
boundary are available.

%s6.2 #&#
\subsection{Discrete approximation of a quasidisk}\label{sec6.2}
A quasicircle is the image of the unit circle under a quasiconformal
mapping. A quasidisk is a (bounded) domain bounded by a quasicircle.
See \cite{Pom92} for definitions and an overview from a conformal
mapping point of view. A quasicircle is not necessarily rectifiable as
the example of the von Koch snowflake shows.

We find it convenient to use an equivalent but more geometric
definition, namely Ahlfors' three-point condition: the closed Jordan
curve $\partial D$ is a quasicircle if and only if there exists a
constant $A <\infty$ such that for any two points $x,y \in\partial D$
it holds that
%
%e55 #&#
\begin{equation}
\label{3point} \operatorname{diam}J(x,y) \le A|x-y|,
\end{equation}
where $J(x,y) \subset\partial D$ is the arc of smaller diameter
connecting $x$ with $y$. One can consider the smallest such $A$ as a
measure of regularity.
This regularity implies some uniform regularity for the grid-domain
approximation $D_n$ and this allows us to estimate the convergence rate
of $\psi_n$ using a result from \cite{W}. See also Section~5 of \cite
{MR} where similar questions are discussed.

%le6.1 #&#
\begin{Lemma}\label{dec15.2}
Let $D$ be a quasidisk satisfying (\ref{3point}) and let $D_n$ be the
$n^{-1} \mathbb{Z}^2$ grid-domain approximation of $D$. Let $\psi,
\psi
_n$ be the normalized conformal maps from $D$ and $D_n$, respectively,
onto $\mathbb{D}$. Then there exists a constant $c < \infty$ depending
only on $A$ and the diameter of $D$ such that
%
%e56 #&#
\begin{equation}
\label{dec8.1} \sup_{z \in D_n}\bigl|\psi_n(z) - \psi(z)\bigr|
\le c \frac{\log n}{\sqrt{n}}.
\end{equation}
\end{Lemma}

\begin{pf}
We will first show that $D_n$ satisfies (\ref{3point}) uniformly in $n$
with a constant $A'$ depending only on $A$. Let $x,y \in\partial D_n$.
First, we consider the case when $|x-y| < 1/n$. Then since $\partial
D_n$ is a Jordan curve which is a subset of the edge set of~$n^{-1}\mathbb{Z}^2$, we have that $\operatorname{diam}J(x,y) \le
\sqrt{2}  |x-y|$.
Now assume that $|x-y| \ge1/n$. Let $\xi$ and $\eta$ be points on
$\partial D$ closest to $x$ and $y$, respectively. Clearly, $|x-\xi|$
and $|y - \eta|$ are both at most $\sqrt{2}/n$. Let $\alpha, \beta$ be
the two line segments connecting $x$ with $\xi$ and $y$ with $\eta$.
First, assume that the curve $\Gamma=J(x,y) \cup\alpha\cup\beta$
separates $J(\xi, \eta)$ from $0$ in $D$. Let $Q_j,   j=1, \ldots, N$,
be those lattice squares whose faces are outside of~$D_n$ but whose
boundaries touch $J(x,y)$. By the construction of~$D_n$ and the Jordan
curve theorem, since $\Gamma$ separates $0$ from $J(\xi, \eta)$, each
$Q_j$ is intersected by $\alpha\cup\beta\cup J(\xi, \eta)$. Consequently,
\[
\operatorname{diam}\Gamma\le\operatorname{diam}J(\xi, \eta) + 2\sqrt{2}/n \le A|
\xi-\eta| + 2\sqrt{2}/n.
\]
Hence,
\[
\operatorname{diam}J(x,y) \le\operatorname{diam}\Gamma\le A|x-y| + (2A +2)
\sqrt{2}/n.
\]
Now, if $\Gamma$ does not separate $J(\xi, \eta)$ from $0$ in $D$, then
since $\Gamma$ is a crosscut of $D$, $(\partial D_n \setminus J(x,y))
\cup\alpha\cup\beta$ does separate $J(\xi, \eta)$ from $0$ in $D$.
Thus, in this case we can do the same argument as in the previous
paragraph showing that $\operatorname{diam} (\partial D_n
\setminus J(x,y)
) \le\operatorname{diam}J(\xi, \eta) + 2\sqrt{2}/n$. But by
definition, $\operatorname{diam}
J(x,y) \le\operatorname{diam} (\partial D_n \setminus
J(x,y) ) $.

Using also the estimate we obtained in the case when $|x-y| < 1/n$, we
conclude that
%
%e57 #&#
\begin{equation}
\label{dec8.2} \operatorname{diam}J(x,y) \le\bigl(A+(2A+2)\sqrt{2}\bigr)|x-y|.
\end{equation}
By (\ref{dec8.2}), there is a constant $c$ depending only on $A$ and
the diameter of $D$ such the Warschawshi structure moduli $\eta
_{W}^{(n)}$ of $\partial D_n$ satisfy
\[
\eta_{W}^{(n)}(\delta) \le c \delta, \qquad\delta\le1.
\]
Consequently, since $D_n \subset D$ and each point on $\partial D_n$ is
within distance $\sqrt{2}/n$ of a~point on $\partial D$, part (a) of
Theorem~VII in \cite{W} implies (\ref{dec8.1}).
\end{pf}

%The same proof works when $D$ is a John domain.
For simplicity, we will now assume that $\partial D$ is $C^{1+\alpha}$
for some $\alpha>0$, that is, we assume that there is a
parameterization of $\partial D$ which has a H\"older-$\alpha$
derivative. By Kellogg's theorem; see, for example, \cite{GM}, this
assumption implies that the conformal map $\psi\dvtx  D \to\mathbb{D}$
(and $\psi^{-1}$) is in $C^{1+\alpha}(\overline{D})$. (So we can take
the conformal parameterization of $\partial D$.) In particular, $\psi$
is bilipschitz on $\overline{D}$, that is, there is a constant $c<
\infty$ depending only on $\alpha$ and the diameter of $D$ such that
%
%e58 #&#
\begin{equation}
\label{kellogg} c^{-1} |z-w| \le\bigl|\psi(z)-\psi(w)\bigr| \le c |z-w|, \qquad z,w
\in \overline{D}.
\end{equation}
Similar uniform estimates, but of H\"older type, and corresponding
versions of Lemma~\ref{dec15.1} (stated again below) hold if $D$ is
assumed to be a quasidisk. Indeed, the uniformizing conformal map and
its inverse are then H\"older continuous on a~neighborhood of $\partial
D$ with an exponent depending only on $A$; see \cite{Pom92}.
From~(\ref{kellogg}), we immediately get the required control over
distortion of annuli up to constants on sufficiently large scales. We
can now prove Lemma~\ref{dec15.1} which we state again.

%le6.2 #&#
\begin{Lemma}
Suppose $D \ni0$ is a simply connected domain Jordan domain with
$C^{1+\alpha}$ boundary, where $\alpha> 0$. Let $D_n$ be the $n^{-1}
\mathbb{Z}^2$ grid-domain approximation of $D$ and let $\gamma_n$ be
a Loewner
curve in $D_n$ connecting $\partial D_n$ with $0$. There is a constant
$c$ depending only on $\alpha$ and the diameter of $D$ such that the
following holds. Set $0< r <1/2$ and $d_n=n^{-r}$ and let $\eta
_{\mathrm{tip}
}^{(n)}(\delta; D_n)$ be the tip structure modulus for $\gamma_n$ in
$D_n$. Then for all $n$ sufficiently large (independently of $\gamma
_n$) the tip structure modulus $\eta_{\mathrm{tip}}^{(n)}(\delta; \mathbb{D})$
for $\psi_n(\gamma_n)$ in $\mathbb{D}$ satisfies
\[
\eta_{\mathrm{tip}}^{(n)}\bigl( c^{-1}  d_n;
\mathbb{D}\bigr) \le c \eta _{\mathrm{tip}}^{(n)}(d_n;
D_n).
\]
\end{Lemma}

\begin{pf}
Let $\eta_n = \eta^{(n)}(d_n; D_n)$. We can assume that $\eta_n \ge2d_n$.
It is enough to verify that there exists a constant $c$ independent of
$n$ such that for all annuli $\mathcal{A}(z)=\{w\dvtx  d_n \le|w-z| \le
\eta
_n\},   z \in D_n$ we have
\[
\psi_n \bigl(\mathcal{A}(z) \cap D_n \bigr) \subset
\bigl\{w\dvtx  c^{-1} d_n \le \bigl|w-\psi_n(z)\bigr| \le c
\eta_n \bigr\} \cap\mathbb{D}.
\]
But this follows immediately from Lemma~\ref{dec15.2} with the
assumption that $d_n$ decays slower than $O(n^{-1/2})$ and (\ref{kellogg}).
\end{pf}
\end{appendix}

%s1.3 #&#
\section*{Acknowledgements}\label{sec1.3}
%Support from the Simons Foundation, Institut Mittag--Leffler, and the
%AXA Research Fund is gratefully acknowledged.
I wish to thank Dmitry
Belyaev, Don Marshall and Steffen Rohde for inspiring and helpful
conversations on the topics of this paper, and Julien Dub\'edat and
Alan Sola for their useful comments on the manuscript. I also wish to
thank the referee for his/her careful reading and valuable comments.

%suskaldyti doi

% imsref loaded by linak, 2014-02-03 09:39:11
% imsref loaded by linak, 2014-02-03 09:52:32

\printaddresses

\end{document}